\newif\ifrendertodo
\rendertodotrue

\documentclass[twoside,11pt]{article}

\usepackage{graphicx}
\usepackage[section]{placeins}
\usepackage[toc,page]{appendix}
\usepackage[backend=biber,style=nature, maxbibnames=3]{biblatex}
\usepackage{enumitem}
\usepackage{color}
\usepackage{colordvi}
\usepackage{mathrsfs}
\usepackage{multirow}
\usepackage{tikz}
\usetikzlibrary{cd}
\usepackage{caption}
\usepackage{subcaption}
\usepackage{lineno}
\usepackage{enumitem}
\def\layersep{1.5cm}
\usepackage{authblk}

\usepackage{amssymb}
\usepackage{amsthm}
\usepackage{amsmath}

\usepackage{mathtools}

\usepackage{bbm}

\usepackage{xcolor}

\usepackage{thmtools, thm-restate}

\usepackage[hidelinks]{hyperref}

\usepackage{cleveref}

\declaretheorem[name=Theorem,
refname={theorem,theorems},
Refname={Theorem,Theorems}
, numberwithin=section]{thm}

\declaretheorem[name=Lemma,
refname={lemma,lemmas},
Refname={Lemma,Lemmas}, sibling=thm]{lem}

\declaretheorem[name=Definition,
refname={definition,definitions},
Refname={Definition,Definitions}, sibling=thm]{definition}

\declaretheorem[name=Remark,
refname={remark, remarks},
Refname={Remark, Remarks}, sibling=thm]{remark}

\declaretheorem[name=Assumption,
refname={assumption, assumptions},
Refname={Assumption, Assumptions}, sibling=thm]{assumption}

\crefname{thm}{theorem}{theorems}
\Crefname{thm}{Theorem}{Theorems}

\crefname{lem}{lemma}{lemmas}
\Crefname{lem}{Lemma}{Lemmas}

\crefname{prop}{proposition}{propositions}
\Crefname{prop}{Proposition}{Propositions}

\crefname{cor}{corollary}{corollaries}
\Crefname{cor}{Corollary}{Corollaries}

\crefname{definition}{definition}{definitions}
\Crefname{definition}{Definition}{Definitions}

\crefname{conj}{conjecture}{conjectures}
\Crefname{conj}{Conjecture}{Conjectures}

\crefname{notation}{notation}{notations}
\Crefname{notation}{Notation}{Notations}

\crefname{remark}{remark}{remarks}
\Crefname{remark}{Remark}{Remarks}

\crefname{example}{example}{examples}
\Crefname{example}{Example}{Examples}

\crefname{counterexample}{counterexample}{counterexamples}
\Crefname{counterexample}{Counterexample}{Counterexamples}

\crefname{assumption}{assumption}{assumptions}
\Crefname{assumption}{assumption}{Assumptions}

\crefname{thm_app}{theorem}{theorems}
\Crefname{thm_app}{Theorem}{Theorems}

\crefname{lem_app}{lemma}{lemmas}
\Crefname{lem_app}{Lemma}{Lemmas}

\crefname{prop_app}{proposition}{propositions}
\Crefname{prop_app}{Proposition}{Propositions}

\crefname{cor_app}{corollary}{corollaries}
\Crefname{cor_app}{Corollary}{Corollaries}

\crefname{definition_app}{definition}{definitions}
\Crefname{definition_app}{Definition}{Definitions}

\crefname{conj_app}{conjecture}{conjectures}
\Crefname{conj_app}{Conjecture}{Conjectures}

\crefname{notation_app}{notation}{notations}
\Crefname{notation_app}{Notation}{Notations}

\crefname{remark_app}{remark}{remarks}
\Crefname{remark_app}{Remark}{Remarks}

\crefname{example_app}{example}{examples}
\Crefname{example_app}{Example}{Examples}

\crefname{counterexample_app}{counterexample}{counterexamples}
\Crefname{counterexample_app}{Counterexample}{Counterexamples}

\crefname{assumption_app}{assumption}{assumptions}
\Crefname{assumption_app}{assumption}{assumptions}

\newcommand{\eps}{\varepsilon}

\newcommand{\RR}{\mathbb{R}}
\newcommand{\PP}{\mathbb{P}}

\newcommand{\EE}{\mathbb{E}}
\newcommand{\II}{\mathbbm{1}}

\newcommand{\Inner}[2]{\left\langle #1 , #2 \right\rangle}

\newcommand{\floor}[1]{\left\lfloor #1\right\rfloor}
\newcommand{\ceil} [1]{\left\lceil  #1\right\rceil}

\DeclareMathOperator{\dd}{\text{d}}

\begin{document}

\title{Uncertainty Quantification for nonparametric regression using Empirical Bayesian neural networks}

\author[1]{Stefan Franssen}
\author[2]{Botond Szab\'o}
\affil[1]{Delft Institute of Applied Mathematics}
\affil[2]{Department of Decision Sciences, Bocconi University\\
Bocconi Institute for Data Science and Analytics (BIDSA)
}



\maketitle

\begin{abstract}
We propose a new, two-step empirical Bayes-type of approach for neural networks. We show in context of the nonparametric regression model that the procedure (up to a logarithmic factor) provides optimal recovery of the underlying functional parameter of interest and provides Bayesian credible sets with frequentist coverage guarantees. The approach requires fitting the neural network only once, hence it is substantially faster than Bootstrapping type approaches. We demonstrate the applicability of our method over synthetic data, observing good estimation properties and reliable uncertainty quantification.
\end{abstract}

\section{Introduction}

Deep learning has received a lot of attention over the recent years due to its excellent performance in various applications, including personalized medicine~\cite{cirilloBigDataAnalytics2019}, self driving cars~\cite{s.ramosDetectingUnexpectedObstacles2017,raoDeepLearningSelfdriving2018}, financial institutions~\cite{huangDeepLearningFinance2020} and estimating power usage in the electrical grid~\cite{liangDeepLearningBasedPower2019,kimElectricEnergyConsumption2019}, just to mention a few. By now it is considered the state-of-the-art technique for image classification~\cite{krizhevskyImagenetClassificationDeep2012} or speech recognition~\cite{hintonDeepNeuralNetworks2012}.

Despite the huge popularity of deep learning, its theoretical underpinning is still limited, see for instance the monograph \cite{anthonyNeuralNetworkLearning1999} for an overview. In our work we focus on the mathematical statistical aspects of how well feed-forward, multilayer artificial neural networks can recover the underlying signal in the noisy data. When fitting a neural network an activation function has to be selected. The most commonly used activation functions include the sigmoid, hyperbolic tangent, rectified linear unit (ReLU) and their variants. Due to computational advantages and available theoretical guarantees we consider the ReLU activation function in our work. The approximation properties of neural network with ReLU activation function has been investigated by several authors recently. In \cite{mhaskarWhenWhyAre2017, poggioWhyWhenCan2017} it was shown that deep networks with a smoothed version of ReLU can reduce sample complexity and the number of training parameters compared to shallow networks to reach the same approximation accuracy. In the discussion paper \cite{schmidt-hieberNonparametricRegressionUsing2017} oracle risk bounds were derived for sparse neural networks in context of the  multivariate nonparametric regression model. This in turn implies for H\"older regular classes (up to a logarithmic factor) rate optimal concentration rates and under additional structural assumptions (e.g. generalized additive models, sparse tensor decomposition) faster rates preventing the curse of dimensionality. The results of  \cite{schmidt-hieberNonparametricRegressionUsing2017} were extended in different aspects by several authors.
In~\cite{suzukiAdaptivityDeepReLU2018} the more general Besov regularity classes were considered and adaptive estimation rates to the smoothness classes were derived. In~\cite{polsonPosteriorConcentrationSparse2018} Bayesian sparse neural networks were proposed, where sparsity was induced by a spike-and-slab prior, and rate adaptive posterior contraction rates were derived. Finally, in~\cite{michaelkohlerRateConvergenceFully2021} it was shown that the sparsity assumption on the neural network is not essential for the theoretical guarantees and similar results to \cite{schmidt-hieberNonparametricRegressionUsing2017} were derived for dense deep neural networks as well.

Most of the theoretical results focus on the recovery of the underlying signal of interest. However, it is at least as important to quantify how much we can rely on the procedure by providing reliable uncertainty statements. In statistics confidence regions are used to quantify the accuracy and remaining uncertainty of the method in a noisy model. Several approaches have already been proposed for statistical uncertainty quantification for neural networks, including bootstrap methods~\cite{osbandDeepExplorationBootstrapped2016} or ensemble methods~\cite{lakshminarayananSimpleScalablePredictive2017}. These methods are typically computationally very demanding especially for large neural networks. Bayesian methods are becoming also increasingly popular, since beside providing a natural way for incorporating expert information into the model via the prior they also provide built-in uncertainty quantification. The Bayesian counterpart of confidence regions are called credible regions which are the sets accumulating a prescribed, large fraction of the posterior mass. For neural networks various fully Bayesian methods were proposed, see for example~\cite{polsonPosteriorConcentrationSparse2018,wangUncertaintyQuantificationSparse2020}, however
they quickly become computationally infeasible as the model size increases. To speed up the computations variational alternatives were proposed, see for instance~\cite{baiEfficientVariationalInference2020}. An extended overview of machine learning methods for uncertainty quantification can be found in the survey \cite{gawlikowskiSurveyUncertaintyDeep2021}.

Bayesian credible sets substantially depend on the choice of the prior and it is not guaranteed that they have confidence guarantees in the classical, frequentist sense. In fact it is known that credible sets do not always give valid uncertainty quantification in context of high-dimensional and nonparametric models, see for instance \cite{dennisd.coxAnalysisBayesianInference1993,davidfreedmanWaldLectureBernsteinvon1999} and hence their use for universally acceptable uncertainty quantification is not supported in general. In recent years frequentist coverage properties of Bayesian credible sets were investigated in a range of high-dimensional and nonparametric models and theoretical guarantees were derived on their reliability under (from various aspects) mild assumptions, see for instance~\cite{szaboFrequentistCoverageAdaptive2015,castilloBernsteinvonMisesPhenomenon2014,yooSupremumNormPosterior2016,rayAdaptiveBernsteinMises2017, belitserCoverageLocalRadial2017, rousseauAsymptoticFrequentistCoverage2020,monardStatisticalGuaranteesBayesian2021} and references therein. However, we have only very limited understanding of the reliability of Bayesian uncertainty quantification in context of deep neural networks. To the best of our knowledge only (semi-)parametric aspects of the problem were studied so far~\cite{wangUncertaintyQuantificationSparse2020}, but these results do not provide uncertainty quantification on the whole functional parameter of interest. 

In our work we propose a novel, empirical Bayesian approach with (relatively) fast computational time and derive theoretical, confidence guarantees for the resulting uncertainty statements. As a first step, we split the data into two parts and use the first part to train a deep neural network. We then use this empirical (i.e. data dependent) network to define the prior distribution used in our Bayesian procedure. We cut of the last layer of this neural network and take the linear combinations of the output of the previous layer with weights endowed by prior distributions, see the schematic representation of the prior in  \Cref{sec:EBDNN} below. The second part of the data is used to compute the corresponding posterior distribution, which will be used for inference. We study the performance of this method in the nonparametric random design regression model, but in principle our approach is applicable more widely. We derive optimal, minimax $L_2$-convergence rates for recovering the underlying functional parameter of interest and frequentist coverage guarantees for the slightly inflated credible sets. We also demonstrate the practical applicability of our method in a simulation study and verify empirically the asymptotic theoretical guarantees.

The rest of the paper is organized as follows. We present our main results in Section \ref{sec:main}. After formally introducing the regression model we describe our Empirical Bayes Deep Neural Network (EBDNN) procedure in Section \ref{sec:EBDNN}, list the set of assumptions under which our theoretical results hold in Section \ref{sec: assumptions} and provide the guarantees for the uncertainty quantification in Section \ref{sec:thm:UQ}. In Section \ref{sec:sim} we present a numerical analysis underlining our theoretical findings and providing a fast and easily implementable algorithm. The proofs are deferred to the Appendix. The proofs for the optimal posterior contraction rates and the frequentist coverage of the credible sets are given in Section \ref{sec:proofs:main}. The approximation of the last layer of the neural network with B-splines is discussed in Section \ref{sec:approx:splines} and some relevant properties of B-splines are collected and verified in Section \ref{sec:B_splines}. Finally, general contraction and coverage results, on which we base the proofs in Section \ref{sec:proofs:main}, are given in Section \ref{sec:general:results}.


\section{Main results}\label{sec:main}
We consider in our analysis the multivariate random design regression model, where we observe pairs of random variables $(X_1,Y_1)$,..., $(X_n,Y_n)$ satisfying
\begin{align*}
Y_i=f_0(X_i)+Z_i,\qquad Z_i\stackrel{iid}{\sim}N(0,\sigma^2),\, X_i\stackrel{iid}{\sim}U([0,1]^d), \quad i=1,...,n,
\end{align*}
for some unknown function $f_0\in L^2([0,1]^d)$. We assume that the underlying function $f_0$ belongs to a $\beta$-smooth Sobolev ball $f_0\in S^{\beta}_d(M)$ with known model hyper-parameters $\beta, M,\sigma^2>0$. It is well known that the corresponding minimax $L_2$-estimation rate of $f_0$ is of order $\eps_n=n^{-\beta/(d+2\beta)}$.

We will investigate the behaviour of multilayer neural networks in context of this nonparametric regression model. We propose an empirical Bayes type of approach, which recovers the underlying functional parameter with the (up to a logarithmic factor) minimax rate and provides reliable uncertainty quantification for the procedure.

\subsection{Empirical Bayes Deep Neural Network (EBDNN)}\label{sec:EBDNN}
We start by formally describing deep neural networks and then present our two-step Empirical Bayes approach. A deep neural network of depth $L > 0$ and width $p = (p_0, \dots, p_L)$ is a collection of weights $W = \{W^i | W^i \in \mathbb{R}^{p_{i} \times p_{i-1}} , i = 1,\dots, L\}$, shifts (or biases) $b = \{ b^i | b^i \in \mathbb{R}^{p_i}, i = 1,\dots L-1 \}$ and an activation function $\sigma$. 
  There is a natural correspondence between deep neural networks with this architecture and functions $f_{W, b}(x):\, \mathbb{R}^{p_0} \rightarrow \mathbb{R}^{p_L}$, with recursive formulation $f_{W, b}(x) = W^L H^{L-1}(x)$, where $H^0_j(x) = x_j$ and $H^i_j(x) = \sigma\left( (W^i H^{i-1}(x))_j + b^i_j\right)$, for $j = 1,\dots, p_i$, $i=1,...,L$. Note that the activation function $\sigma$ is not applied in the final iteration. Different types of activation functions are considered in the literature, including sigmoid, hyperbolic tangent, ReLU, ReLU square. In this work we focus on ReLU activation functions, i.e. we take $\sigma(x)=\max(x,0)$.

Neural networks are very-high dimensional objects, with total number of parameters given by $\sum_{i=1}^{L}(p_{i-1}+1)p_i$. Therefore, from a statistical perspective it is natural to introduce some additional structure in the form of sparsity by setting most of the model parameters $W^i_{jk}$ $i=1,...,L$, $j=1,...,p_{i}$, $k=1,...,p_{i-1}$ and $b^i_{j}$, $i=1,...,L$, $j=1,...,p_i$ to zero. Such networks are called sparse, see the formal definition below.

\begin{definition}\label{Def:SparseDNN}
  We call a deep neural network $s$-sparse if the weights $W^i_{jk}$ and the biases $b^i_j$ take values in $[-1,1]$, and at most $s$ of them are nonzero.
\end{definition}

Neural networks without sparsity assumptions are called dense networks and are more commonly used in practice. In our analysis we focus mainly on sparse networks but our method is flexible and can be easily extended to dense networks as well, which direction we briefly discuss in a subsequent section. Furthermore, we introduce boundedness on the neural network mainly for analytical, but also for practical reasons. We assume that $\|f_{W,b} \|_\infty < F$ for a fixed constant $F>0$.

Next we note that in the last iteration of the recursive formulation  $f_{W, b}(x) = W^L H^{L-1}(x)$ we take the linear combination of the functions $H^{L-1}_j(x)$, $j=1,...,p_{L-1}$. These functions take the role of data generated basis functions of the neural network and will play a crucial role in our method.
\begin{definition}\label{def:DNNBasisFunctions}
  We call the collection of functions $\hat\phi_j = H^{L-1}_j$, $j=1,...,p_{L-1}$ the DNN basis functions generated by the neural network.
\end{definition}

We propose a two stage, Empirical Bayes type of procedure. We start by splitting the dataset $\mathbb{D}_n=\big((X_1,Y_1),... (X_n,Y_n)\big)$ into two (not necessarily equal) partition $\mathbb{D}_{n,1}$ and  $\mathbb{D}_{n,2}$.  We use the first dataset $\mathbb{D}_{n,1}$ to train the deep neural network. Then we build a prior distribution on the so constructed neural network and use the second dataset $\mathbb{D}_{n,2}$ to derive the corresponding posterior. More concretely, we cut-off the last layer of the neural network and take the (data driven) DNN basis functions $\hat\phi_j(x) = H^{L-1}_j(x)$, $j=1,...,p_{L-1}$ defined by the nodes in the $(L-1)$th layer. For convenience we use the notation $k=p_{L-1}$ for the number of DNN basis functions. We construct our prior distribution on the regression function by taking the linear combination of the so constructed basis functions and endowing the corresponding coefficients with prior distributions, i.e.
\begin{align}
\hat{\Pi}_k(\cdot)=\sum_{j=1}^{k}w_j \hat\phi_j(\cdot),\qquad w_j\stackrel{iid}{\sim} g,~j=1,...,k,\label{def:EBDNN:prior}
\end{align}
for some distribution $g$. Then the corresponding posterior is derived as the conditional distribution of the functional parameter given the second part of the data set $\mathbb{D}_{n,2}$. Please find below the schematic representation of our Empirical Bayes DNN prior and the corresponding posterior.\\

\begin{tikzpicture}\label{fig:EmpBayesNN}
    \tikzstyle{annot} = [text width=4em, text centered]
    
    \node[annot] (Datanode) at (0,0) {Data $\mathbb{D}_n$};
    \node[annot] (1sthalfdatanode) at (   \layersep,  2) {First part of $\mathbb{D}_n$: $\mathbb{D}_{n,1}$};
    \node[annot] (2ndhalfdatanode) at (   \layersep, -2) {Second part of $\mathbb{D}_n$: $\mathbb{D}_{n,2}$};
    \node[annot] (DNNnode)         at (3* \layersep,  2) {DNN};
    \node[annot] (priornode)       at (3* \layersep,  0) {Prior : $w_j \stackrel{iid}{\sim} g$, $j=1,...,k$};
    \node[annot] (Basisnode)       at (5* \layersep,  2) {$ \{ \hat{\phi}_j \}_{j = 1}^{k}$};
    \node[annot] (Priorfuncnode)   at (5* \layersep,  0) {$ \hat{\Pi}_k(\cdot)=\sum_{j = 1}^{k} w_j \hat{\phi}_j(\cdot)$};
    \node[annot] (Posteriornode)   at (5* \layersep, -2) {Posterior $\hat{\Pi}_k\left(\cdot\middle| \mathbb{D}_{n,2} \right)$};
    \path[->] (Datanode)        edge (1sthalfdatanode);
    \path[->] (Datanode)        edge (2ndhalfdatanode);
    \path[->] (1sthalfdatanode) edge(DNNnode);
    \path[->] (DNNnode)         edge (Basisnode);
    \path[->] (priornode)       edge (Priorfuncnode);
    \path[->] (Basisnode)       edge (Priorfuncnode);
    \path[->] (2ndhalfdatanode) edge (Posteriornode);
    \path[->] (Priorfuncnode)   edge (Posteriornode);
\end{tikzpicture}

\noindent We note, that often a pre-trained deep neural network is available corresponding to the regression problem of interest. In this case one can simply use that in stage one and compute the posterior based on the whole data set $\mathbb{D}_n$.

\subsection{Assumptions on the EBDNN prior}\label{sec: assumptions}

We start by discussing the deep neural network produced in step one using the first dataset $\mathbb{D}_{n,1}$. As mentioned earlier we consider sparse neural networks following \cite{schmidt-hieberNonparametricRegressionUsing2017}, but our results can be naturally extended to dense network as well. In \cite{schmidt-hieberNonparametricRegressionUsing2017,suzukiAdaptivityDeepReLU2018} optimal minimax concentration rates were derived for sparse neural networks under the assumptions that the networks are $s = k \log(n)$ sparse
and have width $p = \left(d, 6 k, \dots, 6k, k, 1 \right)$, with $k=k_n=n^{d/(d+2\beta)}$. We also apply these assumptions in our approach. However, since uncertainty quantification is a more complex task than estimation we need to introduce some additional structural requirements to our neural network framework. 

One of the big advantage of deep neural networks is that they can learn the best fitting basis functions
to the underlying structure of the  functional parameter $f_0$ of interest, often resulting sharper recovery rates than using standard, fixed bases, see for instance \cite{schmidt-hieberNonparametricRegressionUsing2017,suzukiAdaptivityDeepReLU2018}. However, neural networks in general are highly flexible due to the high-dimensional structure and do not provide a unique, for our goals appropriate representation. For instance, let us consider a neural network with ReLU activation function, see Figure \ref{fig:neuralnetwork} for schematic representation. Then let us include an additional layer in the network before the output layer consisting only one node, see Figure \ref{fig:neuralnetwork:extra}. This node takes the place of the output layer in the original network and since there is only one node in the so constructed last layer the output of the new network is the same as the original one (given that the output function is non-negative). Using the second neural network for our empirical Bayes prior is clearly sub-optimal as we end up with a one dimensional parametric prior for a nonparametric problem, resulting in overly confident uncertainty quantification. Therefore to avoid such pathological cases we introduce some additional structure to our neural network. We assume that the neural network produces nearly orthogonal basis functions, see the precise definition below.

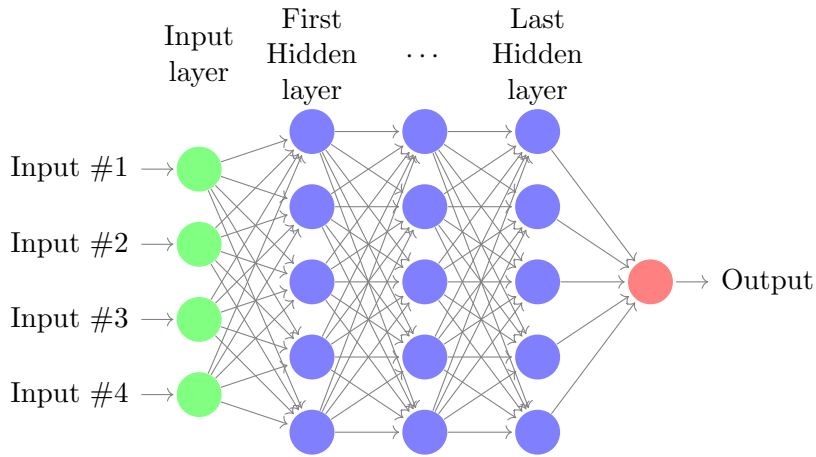
\begin{figure}
\begin{tikzpicture}[shorten >=1pt,->,draw=black!50, node distance=\layersep]
    \tikzstyle{every pin edge}=[<-,shorten <=1pt]
    \tikzstyle{neuron}=[circle,fill=black!25,minimum size=17pt,inner sep=0pt]
    \tikzstyle{input neuron}=[neuron, fill=green!50];
    \tikzstyle{output neuron}=[neuron, fill=red!50];
    \tikzstyle{first hidden neuron}=[neuron, fill=blue!50];
    \tikzstyle{middle hidden neuron}=[neuron, fill=blue!50];
    \tikzstyle{last hidden neuron}=[neuron, fill=blue!50];
    \tikzstyle{annot} = [text width=4em, text centered]

    \foreach \name / \y in {1,...,4}
        \node[input neuron, pin=left:Input \#\y] (I-\name) at (0,-\y) {};

    \foreach \name / \y in {1,...,5}
        \path[yshift=0.5cm]
            node[first hidden neuron] (HFirst-\name) at (\layersep,-\y cm) {};

    \foreach \name / \y in {1,...,5}
        \path[yshift=0.5cm]
            node[middle hidden neuron] (HMiddle-\name) at (2*\layersep,-\y cm) {};

    \foreach \name / \y in {1,...,5}
        \path[yshift=0.5cm]
            node[last hidden neuron] (HLast-\name) at (3*\layersep,-\y cm) {};

    \node[output neuron,pin={[pin edge={->}]right:Output}, right of=HLast-3] (O) {};

    \foreach \source in {1,...,4}
        \foreach \dest in {1,...,5}
            \path (I-\source) edge (HFirst-\dest);

    \foreach \source in {1,...,5}
        \foreach \dest in {1,...,5}
            \path (HFirst-\source) edge (HMiddle-\dest);

    \foreach \source in {1,...,5}
        \foreach \dest in {1,...,5}
            \path (HMiddle-\source) edge (HLast-\dest);
            
    \foreach \source in {1,...,5}
        \path (HLast-\source) edge (O);
        
    \node[annot,above of=HFirst-1, node distance=1cm] (hl) {First Hidden layer};
    \node[annot,above of=HMiddle-1, node distance=1cm] (h2) {$\cdots$};
    \node[annot,above of=HLast-1, node distance=1cm] (h3) {Last Hidden layer};
    \node[annot,left of=hl] {Input layer};
\end{tikzpicture}
\caption{Neural Network with $L-1$ hidden layer.}\label{fig:neuralnetwork}
\end{figure}

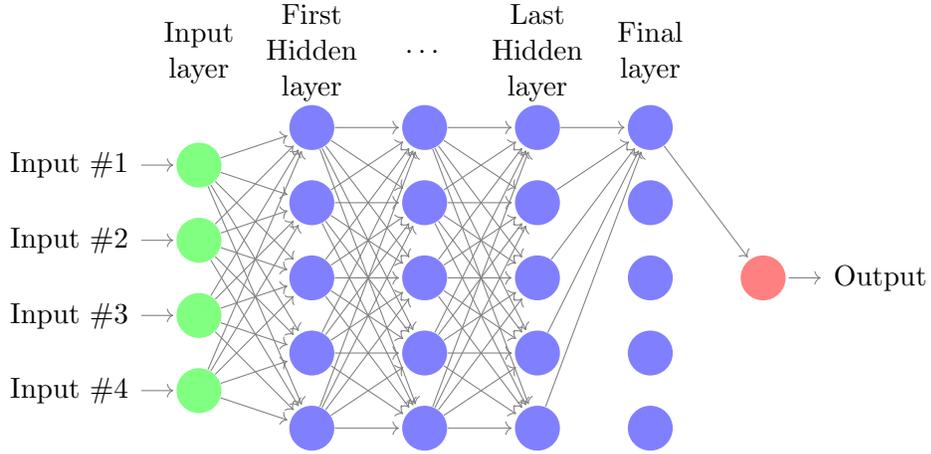
\begin{figure}
  \begin{tikzpicture}[shorten >=1pt,->,draw=black!50, node distance=\layersep]
    \tikzstyle{every pin edge}=[<-,shorten <=1pt]
    \tikzstyle{neuron}=[circle,fill=black!25,minimum size=17pt,inner sep=0pt]
    \tikzstyle{input neuron}=[neuron, fill=green!50];
    \tikzstyle{output neuron}=[neuron, fill=red!50];
    \tikzstyle{first hidden neuron}=[neuron, fill=blue!50];
    \tikzstyle{middle hidden neuron}=[neuron, fill=blue!50];
    \tikzstyle{last hidden neuron}=[neuron, fill=blue!50];
    \tikzstyle{final hidden neuron}=[neuron, fill=blue!50];
    \tikzstyle{annot} = [text width=4em, text centered]

    \foreach \name / \y in {1,...,4}
        \node[input neuron, pin=left:Input \#\y] (I-\name) at (0,-\y) {};

    \foreach \name / \y in {1,...,5}
        \path[yshift=0.5cm]
            node[first hidden neuron] (HFirst-\name) at (\layersep,-\y cm) {};

    \foreach \name / \y in {1,...,5}
        \path[yshift=0.5cm]
            node[middle hidden neuron] (HMiddle-\name) at (2*\layersep,-\y cm) {};

    \foreach \name / \y in {1,...,5}
        \path[yshift=0.5cm]
            node[last hidden neuron] (HLast-\name) at (3*\layersep,-\y cm) {};

    \foreach \name / \y in {1,...,5}
        \path[yshift=0.5cm]
            node[final hidden neuron] (Hfinal-\name) at (4*\layersep,-\y cm) {};

    \node[output neuron,pin={[pin edge={->}]right:Output}, right of=Hfinal-3] (O) {};

    \foreach \source in {1,...,4}
        \foreach \dest in {1,...,5}
            \path (I-\source) edge (HFirst-\dest);

    \foreach \source in {1,...,5}
        \foreach \dest in {1,...,5}
            \path (HFirst-\source) edge (HMiddle-\dest);

    \foreach \source in {1,...,5}
        \foreach \dest in {1,...,5}
            \path (HMiddle-\source) edge (HLast-\dest);
            
    \foreach \source in {1,...,5}
        \path (HLast-\source) edge (Hfinal-1);

    \path (Hfinal-1) edge (O);    
    \node[annot,above of=HFirst-1, node distance=1cm] (hl) {First Hidden layer};
    \node[annot,above of=HMiddle-1, node distance=1cm] (h2) {$\cdots$};
    \node[annot,above of=HLast-1, node distance=1cm] (h3) {Last Hidden layer};
    \node[annot,above of=Hfinal-1, node distance=1cm] (h4) {Final layer};
    \node[annot,left of=hl] {Input layer};
\end{tikzpicture}
\caption{The modified Neural Network of Figure \ref{fig:neuralnetwork} by adding an additional layer before the output layer. This deep neural network provides the same output (given that the output function is non-negative), but in the last hidden layer it has only one node hence provides only one basis function.}\label{fig:neuralnetwork:extra}

\end{figure}

\begin{definition} \label{def:DNNBasisorthonormal}
  We say a neural network produces near orthogonal basis if the Gram matrix $\Sigma_k$ given by $(\Sigma_k)_{i,j} = \langle \hat\phi_i, \hat\phi_j \rangle_2$ satisfies that $ c_1 \II_{k} \leq \Sigma_k \leq c_2 \II_{k}$ for some $0<c_1<c_2<\infty$ and for all $1\leq i,j\leq k$. 
\end{definition}

The requirement of near orthogonality is essential for our analysis to appropriately control the small ball probabilities of the prior distribution which is of key importance in Bayesian nonparametrics. Nevertheless, in view of the simulation study, given in Section \ref{sec:sim} it seems that this assumption can be relaxed. In the numerical analysis section we do not impose this requirement on the algorithm and still get in our examples accurate recovery and reliable uncertainty quantification. We summarize the above assumptions below.

\begin{assumption}\label{ass:dnn}
  Let us take $k=k_n = n^{d/(d+2 \beta)}$ and assume that the neural network $\hat{f}_n$ constructed in step 1 is
  \begin{itemize}
    \item bounded in supnorm $\| \hat{f}_n \|_\infty \leq F$,  
    \item $s = k \log(n)$ sparse, 
    \item has depth $L = \log(n) \ceil{\log_2(\max(4\beta, 4d))}$
    \item has width $p = \left(n, 6 k, \dots, 6k, k, 1 \right)$,
    \item there exists a $C > 0$ independent of $n$ such that the DNN basis functions $\hat\phi = (\hat\phi_1,...,\hat\phi_{k})$ satisfy $\| \theta^T \hat{\phi} \|_\infty \leq C \sqrt{k} \|\theta \|_\infty$,
    \item and the corresponding Gram matrix $\Sigma_k$, given by $(\Sigma_k)_{i,j} = \langle \hat\phi_i, \hat\phi_j \rangle_2$, is nearly orthogonal for some $0 < c_1  < c_2 < \infty$.
  \end{itemize}
\end{assumption}

We denote the class of deep neural networks satisfying \Cref{ass:dnn} by $\mathcal{F}(L, p, s, F, C, c_1, c_2)$. In Appendix \ref{sec:approx:splines} we show that such kind of DNN basis $\hat\phi_1,...,\hat\phi_{k}$ can be constructed. Next assume, similarly to~\cite{schmidt-hieberNonparametricRegressionUsing2017,suzukiAdaptivityDeepReLU2018}, that a near minimizer of the neural network can be obtained. This assumption is required to derive guarantees on the generalisation error of the network produced in step 1.

\begin{assumption}
  \label{ass:Near_minimizer}
  We assume that the network trained in step 1 is a near minimizer in expectation. Let $f_0 \in S_d^{\beta}(M) \cap L_\infty(M)$ and $\epsilon_n = n^{- \beta/(2 \beta + d)}$, then the estimator $\hat{f}_n$ resulting from the neural network satisfies that
  \[
    \EE_{f_0}\left[ \frac{1}{n} \sum_{i = 1}^n \left( y_i - \hat{f}_n(X_i) \right)^2 - \inf_{f \in \mathcal{F}(L, p, s, F, C, c_1, c_2)} \frac{1}{n} \sum_{i = 1}^n \left( Y_i - f(X_i) \right)^2 \right] \leq \epsilon_n \log(n)^3.
  \]
\end{assumption}


It remained to discuss the choice of the prior distribution $g$ on the coefficients of the DNN basis functions. In general we have a lot of flexibility in choosing $g$, but for analytical convenience we assume to work with continuous positive densities.
\begin{assumption}\label{ass:bounded_prior}
  Assume that the density $g$ in the prior\eqref{def:EBDNN:prior} is continuous and positive. 
\end{assumption}

\begin{remark}We note that our proof requires only that the density is bounded away from zero and infinity on a small neighborhood of (the projection of) the true function $f_0$, hence it is sufficient to require that the density $g$ is bounded away from zero and infinity on a large enough compact interval. Since one can construct a neural network with weights between -1 and 1, approximating the true function well enough, we can further relax our assumption and consider densities $g$ supported on on $[-1,1]$.
\end{remark}

\subsection{Uncertainty quantification with EBDNN}\label{sec:thm:UQ}
Our main goal is to provide reliable uncertainty quantification for the outcome of the neural network. Our two-step Empirical Bayes approach gives a probabilistic solution to the problem which in turn can be automatically used to quantify the remaining uncertainty of the procedure. First we show that the corresponding posterior distribution recovers the underlying functional parameter of interest $f_0$ with the minimax contraction rate $\epsilon_n =  n^{- \beta/(2 \beta + d)}$ up to a logarithmic factor.

\begin{thm}\label{thm:contraction:ebdnn}
  Let $\beta, M > 0$ and assume that the EBDNN prior $\hat\Pi_k$, given in \eqref{def:EBDNN:prior}, satisfies Assumptions \ref{ass:dnn}, \ref{ass:Near_minimizer} and \ref{ass:bounded_prior}. Then the corresponding posterior distribution contracts around the true function $f_0\in S^{\beta}_d(M)$ at the near minimax rate, i.e.
      \[
        \limsup_{n \rightarrow \infty}
        \sup_{f_0 \in S_d^{\beta}(M)\cap L_{\infty}(M)}
        \EE_{f_0}\left(
          \hat\Pi_k\left(
            f: \| f - f_0 \|_2 \geq M_n \log^3 (n) \epsilon_n
            \middle|
            \mathbb{D}_{n,2}
          \right)
        \right)
        = 0,
      \]
      for all $M_n \rightarrow \infty$.
\end{thm}

The proof of the theorem is given in Section \ref{sec:proof:contraction:ebdnn}. We note that one can easily construct estimators from the posterior inheriting the same concentration rate as the posterior contraction rate. For instance one can take the center of the smallest ball accumulating at least half of the posterior mass, see Theorem 2.5 of \cite{ghosalConvergenceRatesPosterior2000}. Furthermore, under not too restrictive conditions, it can be proved that the posterior mean achieves the same near optimal concentration rate as the whole posterior, see for instance page 507 of \cite{ghosalConvergenceRatesPosterior2000} or Theorem 2.3. of \cite{hanOraclePosteriorContraction2021}.

Our main focus is, however, on uncertainty quantification. The posterior is typically visualized and summarized by plotting the credible region $C_\alpha$ accumulating $1-\alpha$ fraction (typically one takes $\alpha=0.05$) of the posterior mass. In our analysis we consider $L_2$-balls centered around an estimator $\hat{f}$ (typically the posterior mean or maximum a posteriori estimator), i.e. 
\begin{align*}
C_{\alpha}=\{f:\, \|f-\hat{f}\|_2\leq r_\alpha\}\qquad\text{satisfying}\qquad \hat\Pi_k( f\in C_{\alpha} | \mathbb{D}_{n,2})=1-\alpha.
\end{align*}
More precisely, in case the posterior distribution is not continuous, then the radius $r_{\alpha}$ is taken to be the smallest such that $ \hat\Pi_k( f\in C_{\alpha} | \mathbb{D}_{n,2})\geq 1-\alpha$ holds. 

However, Bayesian credible sets are not automatically confidence sets. To use them from a frequentist perspective reliable uncertainty quantification we have to show that they have good frequentist coverage, i.e. 
$$\inf_{f_0\in S^{\beta}_d(M)} \PP_{f_0}(f_0\in C_{\alpha})\geq 1-\alpha.$$
In our analysis we introduce some additional flexibility by allowing the credible sets to be blown up by a factor $L_n$, i.e. we consider sets of the form
\begin{align} \label{def: credible}
C_{\alpha}(L_n)=\{f:\, \|f-\hat{f}\|_2\leq L_n r_\alpha\}\quad\text{with}\quad  \hat\Pi_k( f\in C_{\alpha}(1) | \mathbb{D}_{n,2})=1-\alpha.
\end{align}
This additional blow up factor $L_n$ is required as the available theoretical results in the literature on the concentration properties of the neural network are sharp only up to a logarithmic multiplicative term and we compensate for this lack of sharpness by introducing this additional flexibility. Furthermore, in view of our simulation study, it seems that a logarithmic blow up is indeed necessary to provide from a frequentist perspective reliable uncertainty statements, see Section \ref{sec:sim}.

The centering point of the credible sets can be chosen flexibly, depending on the problem of interest. In practice usually the posterior mean or mode is considered for computational and practical simplicity. Our results hold for general centering points under some mild conditions. We only require that the centering point attains nearly the optimal concentration rate. We formalize this requirement below.

Let us denote by $f^*=f_{\theta^*}=(\theta^*)^T\hat{\Phi}_k$, with $\hat\Phi_{k}=\big(\hat\phi_1,...,\hat\phi_{k}(x)\big)$ the DNN basis, the Kullback-Leibler (KL) projection of $f_0$ onto our model $\Theta_k=\{\sum_{j=1}^k \theta_j\hat\phi_j:\, \theta\in\RR^k \}$, i.e. let $\theta^*\in\RR^k$ denote the minimizer of the function $\theta\mapsto KL(f_0,\theta^T \hat\Phi_k)$. We note that the KL projection is equivalent with the $L_2$-projection of $f_0$ to $\Theta_k$ in the regression model with Gaussian noise. We assume that the centering point of the credible set is close to $f^*$. 

\begin{assumption}\label{ass:centering}
 The centering point $\hat \theta$ (i.e. $\hat{f}=f_{\hat{\theta}}=\hat{\theta}^T\hat\Phi_k $) satisfies that for all $\delta>0$ there exists $M_\delta>0$ such that
  \begin{equation} \label{freq}
    \sup_{f_0 \in  S_d^{\beta}(M)\cap L_{\infty}(M)} \PP_{f_0} \left( d_n( f^* , \hat{f})\leq  M_\delta n^{-\beta/(2\beta+d)} \right) \geq 1-\delta. 
  \end{equation}
\end{assumption}

This assumption on the centering point is mild. 
For instance considering the centering point of the smallest ball  accumulating a large fraction (e.g. half) of the posterior mass as the center of the credible ball satisfies this assumption. The posterior mean is another good candidate for appropriately chosen priors.

\begin{thm}\label{thm:coverage:ebdnn}
  Let $\beta, M > 0$ and assume that the EBDNN prior $\hat\Pi_k$, given in \eqref{def:EBDNN:prior}, satisfies Assumptions \ref{ass:dnn}, \ref{ass:Near_minimizer} and \ref{ass:bounded_prior}, and the centering point $\hat{f}$ satisfies Assumption \ref{ass:centering}. Then the EBDNN credible credible balls with inflating factor $L_{\delta,\alpha}\log^3(n)$ have uniform frequentist coverage and near optimal size, i.e. for arbitrary $\delta,\alpha>0$ there exists $L_{\delta, \alpha}>0$ such that 
\begin{align}
        &\liminf_n \inf_{f_0 \in S_d^{\beta}(M)\cap L_{\infty}(M)} \PP_{f_0} (f_0 \in C_{\alpha}(L_{\delta, \alpha}\log^3 n) ) \geq 1 - \delta, \label{eq:UniformCoverageEvenSplitDNN}\\
       &\liminf_n \inf_{f_0\in S_d^{\beta}(M)\cap L_{\infty}(M)} \PP_{f_0}(r_\alpha \leq C n^{-\beta/(2\beta+d)}) \geq 1- \delta,        \label{eq:UniformRadiusEvenSplitDNN}
\end{align}
      for some large enough $C > 0$.
\end{thm}

We defer the proof of the theorem to Section \ref{sec:proof:coverage:ebdnn}.

\section{Numerical Analysis}\label{sec:sim}
So far we have studied the EBDNN methodology from a theoretical, asymptotic perspective. In this section we investigate the finite sample behaviour of the procedure. First note that the theoretical bounds in~\cite{schmidt-hieberNonparametricRegressionUsing2017,suzukiAdaptivityDeepReLU2018} are not known to be tight. Sharper bounds would result in more accurate procedure with smaller adjustments for the credible sets. For these reasons we study the performance of the EBDNN methodology in synthetic data sets, where the estimation and coverage properties can be empirically explicitly evaluated.

In our implementation we deviate for practical reasons in three points from the theoretical assumptions considered in the previous sections. First, in practice sparse deep neural networks are rarely used, as they are typically computationally too involved to train. Instead, dense deep neural networks are applied routinely which we will also adopt in our simulation study.
Moreover, the global optima typically can not be retrieved when training a neural network. The common practice is to use gradient descent and aim for attaining good local minimizer.
Finally, the softwares used to train deep neural networks do not necessarily return a near orthogonal deep neural network. Even worse, some of the produced basis functions can be collinear or even constantly zero. To guarantee that the produced basis functions are nearly orthogonal one can either apply the Gram-Schmidt procedure or introduce a penalty for collinearity. We do not pursue this direction in our numerical analysis, but use standard softwares and investigate the robustness of our procedure with respect to these aspects. So instead of studying our EBDNN methodology under our restrictive assumptions we investigate its performance in more realistic scenarios. We fit a dense deep neural network using standard gradient descent and do not induce sparsity or near orthogonality to our network.

\subsection{Implementation details}

We have implemented our EBDNN method in Python. 
We used Keras~\cite{cholletKeras2015} and tensorflow~\cite{abadiTensorFlowLargescaleMachine2015} to fit a deep neural network using the first half of the data.
We use gradient decent to fit a dense neural network with $L = \ceil{\log_2(\beta)\log_2(n)}$ layers. Each of the first $L-1$ hidden layers has width $6 k_n$, with $k_n = n^{d/(2\beta+d)}$, and we apply the ReLU activation function on them. The last layer has width $k_n$ and the identify map is taken as activation function on it, that is, we take the weighted linear combination of these basis functions.
Then we extract the basis functions by removing the last layer and endow the corresponding weights by independent and identically distributed standard normal random variables, to exploit conjugacy and speed up the computations. We derive credible regions by sampling from the posterior using Numpy~\cite{harrisArrayProgrammingNumPy2020} and empirically computing the quantiles and the posterior mean used as the centering point. The corresponding code is available at \cite{franssenUncertaintyQuantificationUsing}.

\subsection{Results of the numerical simulations}
We consider two different regression functions in our analysis.
\begin{equation*}
  f_1 = \sum_{i = 1}^\infty \frac{ \sin(i) \cos( \pi (i - 0.5) x)}{i^{1.5}}, \qquad f_2 = \sum_{i = 1}^\infty \frac{\sin(i^2)\cos( \pi (i - 0.5) x)}{i^{1.5}}.
\end{equation*}
Note that both of them belong to a Sobolev class with regularity 1. In the implementation we have considered a sufficiently large cut-off of their Fourier series expansion.

In our simulation study we investigate beyond the $L_2$-credible balls also the pointwise and $L_{\infty}$ credible regions as well. In our theoretical studies we have derived good frequentist coverage after inflating the credible balls by a $\log^3(n)$ factor. In our numerical analysis we observe that (at least on a range of examples) a $\log(n)$ blow-up factor is sufficient, while a $\sqrt{\log(n)}$ blow-up is not enough.

We considered sample sizes $n=$ 1000, 5000, 10000, and 50000, and repeated each of the experiments 1000 times. We report in Table \ref{table:L2dist} the average $L_2$-distance and the corresponding standard deviation between the posterior mean and the true functional parameter of interest. 
\begin{center}
\begin{tabular}{l c c c c}
$n$ & 1000 & 5000 & 10000 & 50000 \\
$f_1$ & $213.33 \pm 116.06$ & $120.36 \pm 29.29$ & $96.84 \pm 18.58$ & $48.94 \pm 9.19$\\
$f_2$ & $221.20 \pm 113.48$ & $140.04 \pm 46.69$ & $108.82 \pm 9.73$ & $75.74 \pm 3.60$\\
\end{tabular}
\captionof{table}{Average $L_2$-distance between the posterior mean and the true function based on 1000 repetitions. The sample sizes range from 1000 to 50000.}
\label{table:L2dist}
\end{center}

Furthermore, we investigate the frequentist coverage properties of the EBDNN credible sets by reporting the fraction of times the (inflated) credible balls contain the true function out of the 1000 runs in Table \ref{table:L2coverage}. One can observe that in case of function $f_2$ a $\sqrt{\log n}$ blow up factor is not sufficient and the more conservative $\log n$ inflation has to be applied, which provides reliable uncertainty quantification in both cases.

\begin{center}
\begin{tabular}{l c c c c c}
 function & blow-up & 1000 & 5000 & 10000 & 50000 \\
\multirow{3}{*}{$f_1$} & none & 0.0 & 0.0 & 0.0 & 0.0\\
 & $\sqrt{\log(n)}$ & 0.893 & 0.896 & 0.848 & 0.928\\
 & $\log(n)$ & 0.951 & 0.997 & 1.0 & 1.0\\
 & & & & & \\
\multirow{3}{*}{$f_2$} & none & 0.0 & 0.0 & 0.0 & 0.0\\
 & $\sqrt{\log(n)}$ & 0.834 & 0.693 & 0.736 & 0.003\\
 & $\log(n)$ & 0.934 & 0.986 & 1.0 & 1.0\\
\end{tabular}
\captionof{table}{Frequentist coverage of the inflated $L_2$-credible balls based on 1000 runs of the algorithm. Sample size is ranging between 1000 and 50000 and we considered multiplicative inflation factors between 1 and $\log n$.}
\label{table:L2coverage} 
\end{center}

We also report the size of the $L_2$ credible balls in Table \ref{table:L2width}. One can observe that the radius of the credible balls are substantially smaller than the average Euclidean distance between the posterior mean the true functions $f_1$ and $f_2$ of interest, respectively. This explains the necessity of the inflation factor applied to derive reliable uncertainty quantification from the Bayesian procedures. We illustrate the method in Figure \ref{fig:L2crediblebands}. Note that the true function is inside of the region defined by the convex hull of the 95\% closest posterior draws to the posterior mean.

\begin{center}
\begin{tabular}{l c c c c}
 $n$ & 1000 & 5000 & 10000 & 50000 \\
$f_1$ & $101.91 \pm 15.69$ & $52.27 \pm 3.60$ & $38.45 \pm 2.33$ & $19.47 \pm 1.03$\\
$f_2$ & $91.26 \pm 16.50$ & $49.61 \pm 3.76$ & $37.46 \pm 2.03$ & $19.69 \pm 0.85$\\
\end{tabular}
\captionof{table}{$L_2$ The average diameter and the corresponding standard deviation of the (non-inflated) credible balls based on $1000$ runs of the algorithm. }
\label{table:L2width}
\end{center}

\begin{figure}
  \centering
  \includegraphics[width=\linewidth]{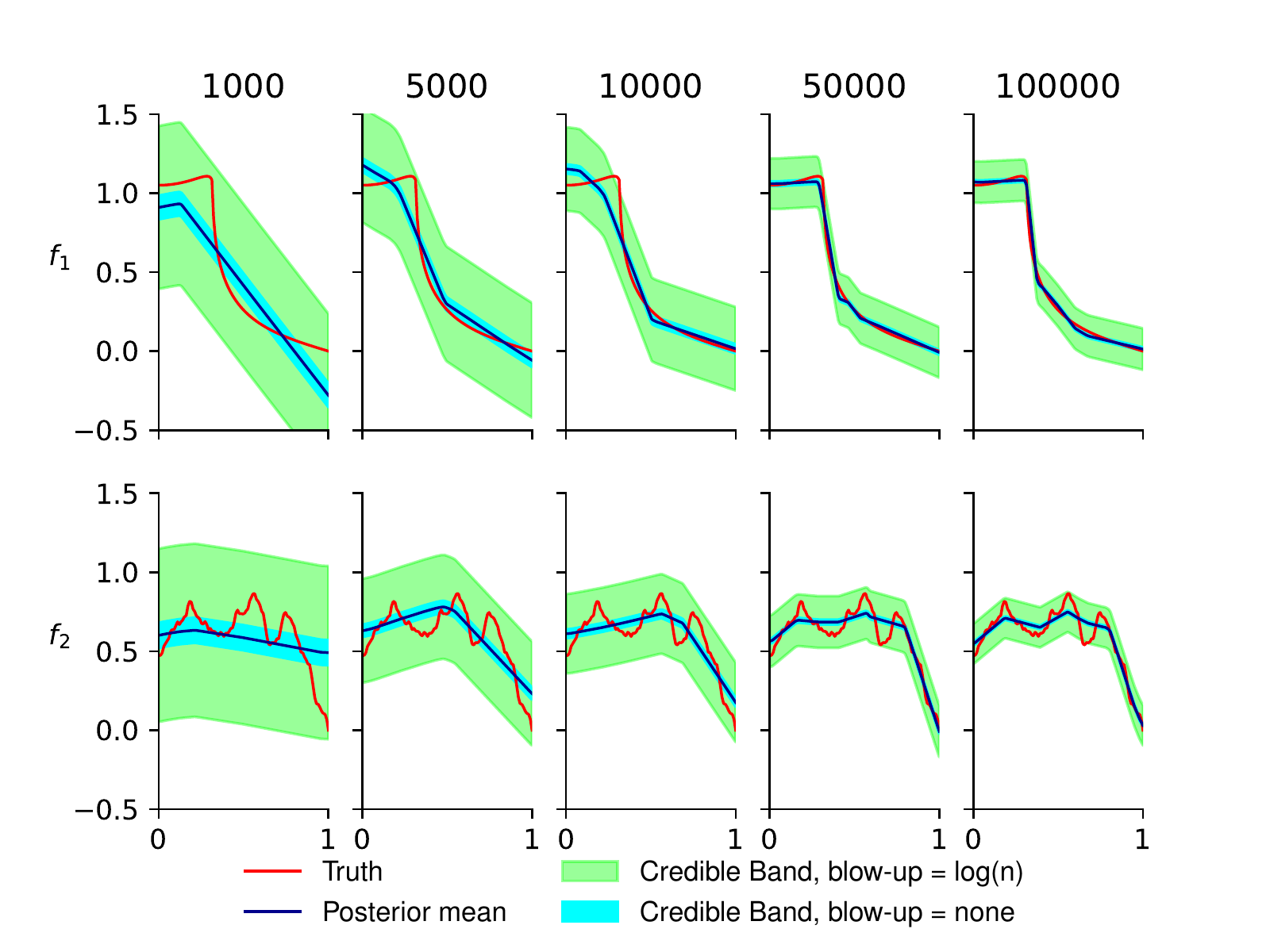}
  \caption{EBDNN $L_2$-credible balls illustrated by the region covered by the 95\% closest draw from the posterior to the posterior mean in $L_2$-distance. Sample size increases from 1000 to 100000. The original credible sets are plotted by light blue and the inflated credible sets by green.}
  \label{fig:L2crediblebands}
\end{figure}

Next we investigate the point wise and $L_{\infty}$ credible regions. Compared to the $L_2$-credible balls we note that the $L_\infty$ credible bands are roughly a factor $2$ wider, see Table \ref{table:Linfwidth}.

\begin{center}
\begin{tabular}{l c c c c}
 $n$ & 1000 & 5000 & 10000 & 50000 \\
$f_1$ & $175.93 \pm 28.33$ & $101.10 \pm 10.66$ & $77.08 \pm 6.50$ & $46.38 \pm 3.83$\\
$f_2$ & $145.60 \pm 30.59$ & $94.65 \pm 9.93$ & $78.41 \pm 5.91$ & $51.26 \pm 3.84$\\
\end{tabular}
\captionof{table}{The average supremum diameters and the corresponding standard deviation of the (non-inflated) credible balls based on $1000$ runs of the algorithm.}
\label{table:Linfwidth}
\end{center}
At the same time, the distance between the posterior mean and the true regression function is roughly a factor $6$ larger compared to the situation in the $L_2$ norm, see Table \ref{table:Linfdist}.

\begin{center}
\begin{tabular}{l c c c c}
$n$ & 1000 & 5000 & 10000 & 50000 \\
$f_1$ & $442.45 \pm 181.02$ & $355.40 \pm 55.67$ & $325.34 \pm 42.26$ & $199.95 \pm 31.32$\\
$f_2$ & $569.47 \pm 119.33$ & $360.34 \pm 81.54$ & $255.59 \pm 36.78$ & $169.60 \pm 11.41$\\
\end{tabular}
\captionof{table}{Average supremum norm-distances between the posterior mean and the true function based on 1000 repetitions. }
\label{table:Linfdist}
\end{center}

This results in worse coverage results than in the $L_2$ case, although inflating the credible bands by a $\log (n)$ factor still results in reliable uncertainty quantification on our simulated data, see Table \ref{table:Linfcoverage}  and Figure \ref{fig:Linftycrediblebands}.

\begin{figure}
  \centering
  \includegraphics[width=\linewidth]{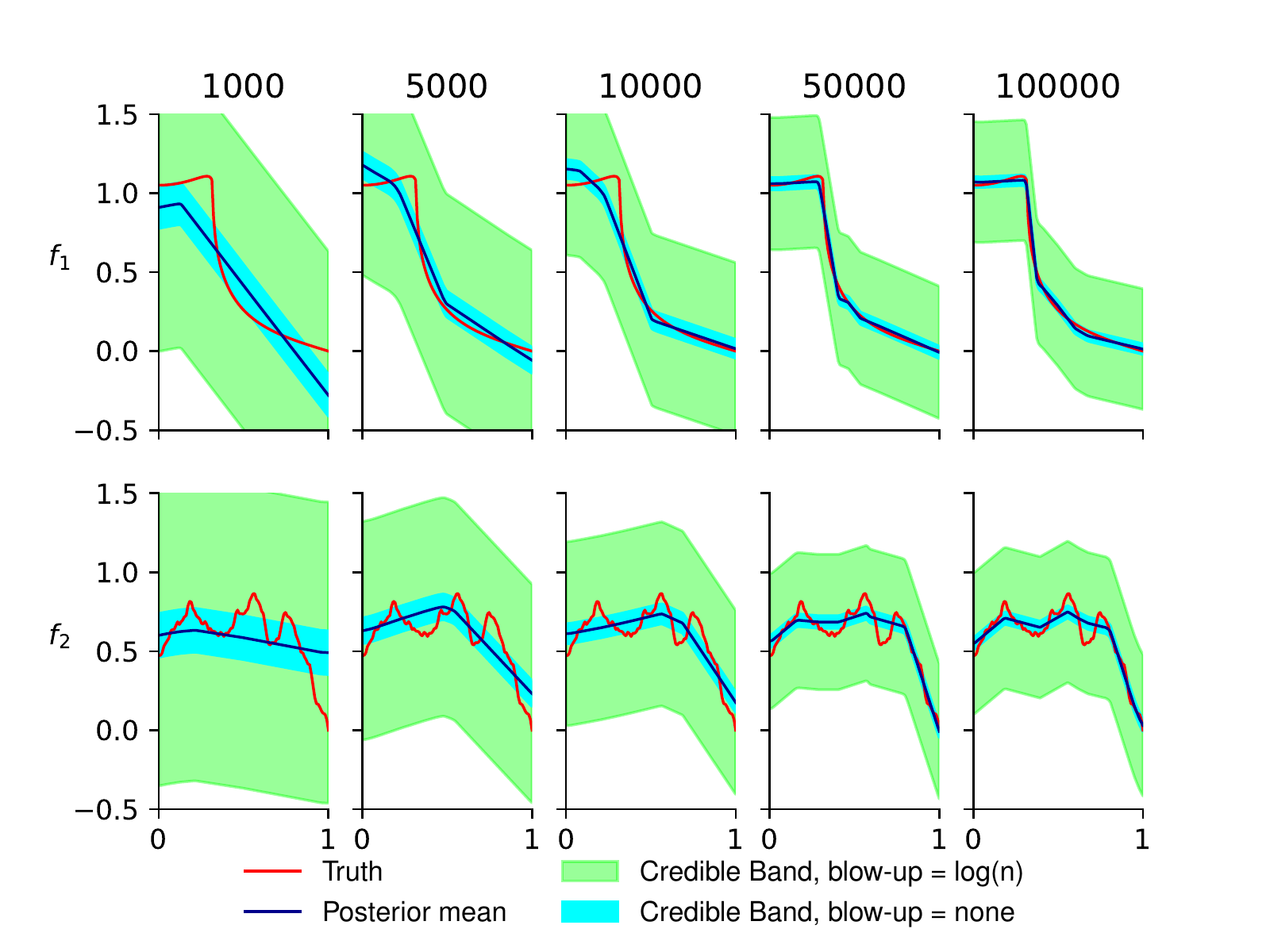}
\caption{EBDNN $L_\infty$-credible regions. The bands are formed by keeping the 95\% closest draws from the posterior to the posterior mean in $L_\infty$-distance. Sample size increases from 1000 to 100000. The original credible bands are plotted by light blue and the inflated credible bands by green.}
  \label{fig:Linftycrediblebands}
\end{figure}

\begin{center}
\begin{tabular}{l c c c c c}
 function & blow-up & 1000 & 5000 & 10000 & 50000 \\
\multirow{3}{*}{$f_1$} & none & 0.0 & 0.0 & 0.0 & 0.0\\
 & $\sqrt{\log(n)}$ & 0.779 & 0.12 & 0.015 & 0.081\\
 & $\log(n)$ & 0.957 & 0.997 & 1.0 & 1.0\\
 & & & & & \\
\multirow{3}{*}{$f_2$} & none & 0.0 & 0.0 & 0.0 & 0.0\\
 & $\sqrt{\log(n)}$ & 0.101 & 0.181 & 0.384 & 0.503\\
 & $\log(n)$ & 0.946 & 0.986 & 1.0 & 1.0\\
\end{tabular}
\captionof{table}{Frequentist coverage of the inflated $L_\infty$-credible balls based on 1000 runs of the algorithm. Sample size is ranging between 1000 and 50000 and we considered multiplicative inflation factors between 1 and $\log n$.}
\label{table:Linfcoverage} 
\end{center}

\section{Discussion}
We have introduced a new methodology, the empirical Bayesian deep neural networks (EBDDN). 
We studied the accuracy of estimation and uncertainty quantification of this method from a frequentist, asymptotic point of view. We have derived optimal contraction rates and frequentist coverage guarantees for the slightly inflated credible balls.

We have studied the practical performance of the EBDDN method on synthetic data sets under relaxed, practically more appealing assumptions compared to the ones required for the theoretical guarantees. The simulation study suggests that a smaller inflation factor can be applied than the one resulting from the theoretical study, but it can not be completely avoided as otherwise the credible sets would provide over-confident, misleading uncertainty statements.

Our approach has two main advantages over competing methodologies. 
First of all, EBDNN has theoretic guarantees, while other methodologies, to the best of our knowledge, do not have theoretical underpinning on the validity of the uncertainty quantification for the functional parameter of interest.
Secondly, EBDNN are easy to compute compared to other approaches, since they only require the training of one deep neural network. 

Moreover, it is easy to adapt this methodology to other frameworks by using a different activation function in the last layer. 
This would for example allow to cover classification by using a softmax activation function in the final layer. A follow up simulation study, based on our results, was executed in the logistic regression model in the master thesis~\cite{casaraSimulationStudyUncertainty2021}.

\section{Acknowledgements}
Initial discussions surrounding this project started in Lunteren 2019, together with Amine Hadji and Johannes Schmidt-Hieber. 
We would like to thank both of them for the fruitful discussions. 
This project has received funding from the European Research Council (ERC) under the European Union’s Horizon 2020 research and innovation programme (grant agreement No 101041064). 
Moreover, the research leading to these results is partly financed by a Spinoza prize awarded by the Netherlands Organisation for Scientific Research (NWO).

\printbibliography


\appendix
\section{Proof of the main results}\label{sec:proofs:main}
Before providing the proof of our main theorems we recall a few notations used throughout the section. We denote by $\mathbb{D}_{n,1}$ and $\mathbb{D}_{n,2}$ the first and second half of the data respectively, i.e.
  \begin{align*}
      \mathbb{D}_{n,1} &= ((X_1, Y_1), \dots, (X_{\floor{\frac{n}{2}}}, Y_{\floor{\frac{n}{2}}})),\\
      \mathbb{D}_{n,2} &= ((X_{\floor{\frac{n}{2}}+1}, Y_{\floor{\frac{n}{2}}+1}), \dots, (X_n, Y_n)).
  \end{align*}
Furthermore, we denote by $f^*$ the $L_2$-projection of $f_0$ into the linear space spanned by the DNN basis based on the first data set $\mathbb{D}_{n,1}$, as it was defined above Assumption \ref{ass:centering}. Next we give the proofs for our main theorems.

\subsection{Proof of Theorem \ref{thm:contraction:ebdnn}}\label{sec:proof:contraction:ebdnn}

First note that by triangle inequality $\| f - f_0 \|_2 \leq \| f - f^* \|_2 + \| f^* - f_0 \|_2$. We deal with the two terms on the right hand side separately.

In view of Lemma \ref{thm:contraction} (with $\mathbb{D}_n:=\mathbb{D}_{n,2}$, $k=n^{d/(2\beta+d)}$,  $\Sigma_k$ defined by the DNN basis functions $\hat\phi_1$,..., $\hat\phi_k$, $\Pi_k=\hat\Pi_k$ and with respect to the conditional distribution given the first data set $\mathbb{P}_{f_0}^{|\mathbb{D}_{n,1}}$)
we get that for every $\delta>0$ there exists $M_\delta<\infty$,
\begin{align}
    \sup_{f_0 \in S_d^{\beta}(M)}
    \EE_{f_0}^{|\mathbb{D}_{n,1}}
    \hat{\Pi}_k\left(f:\,
      \| f - f^* \|_2 \geq M_{\delta} n^{-\beta/(2\beta+d)} |    \mathbb{D}_{n,2}
    \right)\leq\delta,\label{eq:help:contr_dnn}
\end{align}
with $\PP_{f_0}$-probability tending to one. Hence, it remained to deal with the term $\| f^* - f_0 \|_2$. We introduce the event 
$$A_n = \{ \| f^* - f_0 \|_2 \leq  (M_n/2) n^{-\beta/(2\beta+d)} \log^3(n) \},$$
 which is independent from the second half of the data $\mathbb{D}_{n,2}$, hence

  \begingroup
  \allowdisplaybreaks
  \begin{align*}
    &
    \EE_{f_0}
    \hat\Pi_k\left( \theta:\, \|f_{\theta} - f_0 \|_2 \geq M_n n^{-\frac{\beta}{2\beta+d}} \log^3(n) \middle| \mathbb{D}_{n,2} \right) \\
    &\qquad\leq 
    \EE_{f_0}\Big( \II_{A_n}\EE_{f_0}^{|\mathbb{D}_{n,1}}
    \hat\Pi_k( \theta:\, \|f_\theta - f^*\|_2 \geq M_n n^{-\frac{\beta}{2\beta+d}}| \mathbb{D}_{n,2})\Big)+
    \EE_{f_0} \II_{A_n^c}.
\end{align*}
\endgroup
The first term is bounded by $\delta$ in view of assertion \eqref{eq:help:contr_dnn}, while the second term tends to zero in view of Theorem 4 of \cite{suzukiAdaptivityDeepReLU2018} combined with Markov inequality.

\subsection{Proof of Theorem \ref{thm:coverage:ebdnn}}\label{sec:proof:coverage:ebdnn}

Let $L_n = L_{\epsilon, \alpha} \log(n)^3$, $\eps_n=n^{-\beta/(2\beta+d)}$ and $k=n^{d/(2\beta+d)}$. Then by triangle inequality we get
  \begin{align}
    \PP_{f_0} \left( f_0 \in C_{\alpha}(L_n) \right) &= \PP_{f_0} \left( \| f_0 - \hat{f}\|_2 \leq L_n r_\alpha \right)\nonumber \\
    & \geq \PP_{f_0}\Big( \hat\Pi_k(\theta:\, \| f_\theta -\hat{f} \|_2 \leq \| \hat{f} - f_0 \|_2/ L_n | \mathbb{D}_{n,2}) < 1 - \alpha \Big).\label{eq: help1}
    \end{align}
    Furthermore, let us introduce the event $A_n = \{\| \hat{f} - f_0\|_2 \leq  M_{\eps}\log^3 (n)\epsilon_n \}$. Note that by triangle inequality and in view of Assumption \ref{ass:centering} (for large enough choice of $M_{\eps}$) and Theorem 4 of~\cite{suzukiAdaptivityDeepReLU2018} (with $f^*$ denoting the $L_2$-projection of $f_0$ to the linear space spanned by the DNN basis) combined with Markov's inequality,
\begin{align*}
\PP_{f_0}(A_n^c)&\leq \PP_{f_0} \big( \| \hat{f} - f^*\|_2 \geq  M_{\eps}\epsilon_n/2 \big)\\
 &\qquad + \PP_{f_0} \big( \| f_0 - f^*\|_2 \geq  M_{\eps}\log^3 (n)\epsilon_n/2 \big)\\
&\leq \eps/3+\eps/3=(2/3)\eps.
\end{align*}
Hence, the probability on the right hand side of \eqref{eq: help1} is lower bounded by
  \[
    \PP_{f_0}\Big(\hat\Pi_k(\theta:\, \| f_\theta - \hat{f} \|_2 \leq \frac{M_{\eps} \epsilon_n}{L_{\epsilon,\alpha}}  | \mathbb{D}_{n,2} ) < 1 - \alpha\Big)-(2/3)\eps.
  \]

We finish the proof by showing  that the first term in the preceding display is bounded from below by $1-\eps/3$. Since by assumption the DNN basis is nearly orthogonal with $\PP_{f_0}$-probability tending to one, we get in view of Lemma \ref{thm:cov:reg} below (applied with probability measure $\PP_{f_0}^{|\mathbb{D}_{n,1}}$, $k=n^{d/(2\beta+d)}$ and $\Pi_k=\hat\Pi_k$) that for all $\eps>0$ there exists $\delta_{\eps,\alpha}>0$ such that
 \begin{equation*}
 \sup_{f_0\in S^\beta_d (M)} \mathbb{E}_{f_0}^{|\mathbb{D}_{n,1}} \hat\Pi_{k}\Big(f_{\theta}:\,\|f_{\theta}- \hat{f}\|_2\leq \delta_{\eps,\alpha}  \sqrt{k/n}| \mathbb{D}_{n,2} \Big) \leq \eps(1-\alpha)/3,
 \end{equation*}
 with $\PP_{f_0}$-probability tending to one.  Let us take $L_{\epsilon, \alpha} \geq M_\epsilon/\delta_{\epsilon, \alpha}$ and combine the preceding display with Markov's inequality,
\begin{align*}
    \PP_{f_0}&\Big( \hat\Pi_k\big(\theta:\, \|f_\theta- \hat{f}\|_2 \leq \frac{M_\epsilon \epsilon_n}{L_{\eps,\alpha}} | \mathbb{D}_{n,2} \big) \geq 1 - \alpha \Big)\nonumber\\
 &\quad\leq \frac{\EE_{f_0}\Big(\hat\Pi_k\big(\theta:\, \|f_\theta- \hat{f}\|_2 \leq \delta_{\epsilon,\alpha} \epsilon_n|\mathbb{D}_{n,2} \big)\Big)}{1 - \alpha}\leq \eps/3,
\end{align*}
concluding the proof.

\begin{remark}
  We point out that the extra multiplicative term $\log^3(n)$ is the result of the lack of sharpness in the convergence rate of deep neural network estimator $f_{n}^*$. Sharper bounds for this estimation would result in smaller blow up factor.
\end{remark}

We state below that under the preceding assumptions the EBDNN credible sets are valid frequentist confidence sets as well.

\section{Approximation of Splines using Deep neural networks}\label{sec:approx:splines}

This section considers the construction of orthonormal basis $\phi_{1}, \dots, \phi_{k}$ in $d$-dimension using neural networks. We first show that splines can be approximated well with neural networks and then we achieve near orthonormality by rescaling. We summarize the main results in the following lemma.

\begin{lem}\label{Lem:BasisProperties}
  There exist DNN basisfunctions $\phi_{1}, \dots, \phi_{k}$ with $k=n^{d/(d+2\beta)}$ such that
  \begin{itemize}
    \item For every $f_0\in S^{\beta}_d(M) \cap L_\infty(M)$, with $\beta>d/2$, there exists $\theta=(\theta_1,...,\theta_{k})\in\ell_{\infty}(1)$ such that $ \| f_0 - \sum_{j=1}^k \theta_j \phi_{j} \|_{2} < \epsilon_n$ with $\eps_n=n^{-\beta/(d+2\beta)}$.
    \item The rescaled DNN basis functions $\sqrt{k}\phi_{1}, \dots, \sqrt{k}\phi_{k}$ are nearly orthonormal in the sense of Definition~\ref{def:DNNBasisorthonormal}.
    \item The basis functions are bounded in supremum norm, i.e. $\|\phi_{j}\|_{\infty}\leq 2$, $j=1,...,k$.  
  \end{itemize}
\end{lem}

\begin{proof}
In view of Lemma \ref{thm:ApproximationPowerBySplines} there exists $\theta=(\theta_1,...,\theta_{k})\in \mathbb{R}^{k}$ such that $\|f_0-\sum_{j=1}^{k} \theta_j B_{j}\|_2\leq n^{-\beta/(d+2\beta)}$, where $B_j$, $j=1,...,k$ denote the cardinal B-splines of order $q\geq \beta$, see \eqref{eq: B:splines} and teh remark below it about the single index representation. Moreover, if $\|f_0 \|_\infty < M$, then one can choose $\theta\in\mathbb{R}^k$ so that $\| \theta \|_\infty < M$. Furthermore, in view of Proposition 1 of \cite{suzukiAdaptivityDeepReLU2018} one can construct a DNN basis 
\begin{align}
\phi_{1}, \dots, \phi_{k}, \quad\text{such that}\quad \|B_{j}-\phi_{j}\|_{\infty}\leq C/n,\quad j=1,...,k,\label{eq:dnn:basis:Suzuki}
\end{align}
for some universal constant $C>0$. Therefore, by triangle inequality 
\begin{align*}
 \| f_0 - \sum_{j=1}^{k} \theta_j \phi_{j} \|_{2}&\leq \| f_0-\sum_{j=1}^{k} \theta_j B_{j}\|_2+ \| \sum_{j=1}^{k} \theta_j (B_{j}-\phi_j)\|_2 \\
&\lesssim n^{-\beta/(d+2\beta)}+  M k/n\lesssim n^{-\beta/(d+2\beta)}.
\end{align*}
Then in Lemma \ref{lem:SplineAproxNearlyOrthogonal} below we show that the above DNN basis $(\phi_{j})_{j=1,..,k}$ inherits the near orthogonality of B-splines, which is verified for dimension $d$ in Lemma \ref{lem:orthogonality}. The boundedness of the B-splines, will be also inherited by the above DNN basis $(\phi_{j})_{j=1,..,k}$ in view of Lemma~\ref{lem:SplineAproxBounded}.
Moreover, basis can be rescaled in such a way that the coefficients are in the interval $[-1,1]$.
\end{proof}

We provide below the two lemmas used in the proof of the previous statement.

\begin{lem}\label{lem:SplineAproxNearlyOrthogonal}
  The rescaled DNN basis $\sqrt{k}\phi = \left(\sqrt{k}\phi_{1}, \dots, \sqrt{k}\phi_{k}\right)$ given in~\eqref{eq:dnn:basis:Suzuki} is nearly orthonormal in the sense of Definition~\ref{def:DNNBasisorthonormal}.
\end{lem}

\begin{proof}
  In view of Lemma 1 of \cite{suzukiAdaptivityDeepReLU2018} the above DNN basis has the same support as the B-splines of order $q = \ceil{\beta}$. Let us define the matrices $Q_k,R_k\in \mathbb{R}^{k\times k}$ as
\begin{align*}
(Q_{k})_{i,j} &= \Inner{\sqrt{k}B_{i}}{\sqrt{k}B_{j}}= k\int_0^1 B_{i} (x)B_{j}(x) \dd x,\\
(R_{k})_{i,j} &= \Inner{\sqrt{k}\phi_{i}}{\sqrt{k}\phi_{j}} - \Inner{\sqrt{k}B_{i}}{\sqrt{k}B_{j}}= k\int_0^1 \phi_{i}(x) \phi_{j}(x) - B_{i} (x)B_{j}(x) dx,
\end{align*}
for $i,j\in\{1,\dots,k\}$. Then $Q_k + R_k$ is the matrix consisting of the innerproducts in the constructed basis. Note that in view of \eqref{eq:dnn:basis:Suzuki} there exists a constant $C' >0 $ such that $|(R_{k})_{i,j}| < C'k/n$. Furthermore, we note that a B-spline basis function of order $q$ has intersecting support with at most $(2q)^d$ other B-spline basis functions. In view of Lemma 1 of \cite{suzukiAdaptivityDeepReLU2018}, the same holds for the $\phi_j$, $j=1,...,k$ basis. This means that there are at most $(2q)^d$ non-zero terms in every row or column and hence in total we have at most $(2q)^dk$ nonzero cells in the matrix.

Define $(M_{k})_{i,j} = |(R_{k})_{i,j}|$. Then the spectral radius of $M_k$ is an upper bound of the spectral radius of $R_k$ by Wielandt's theorem \cite{weissteinWielandtTheorem}. Since $M_k$ is a nonnegative matrix, in view of the Perron-Frobenius theorem \cite{perronZurTheorieMatrices1907,frobeniusUeberMatrizenAus1912}, the largest eigenvalue in absolute value is bounded by constant times $k^2/n$. Next note that both $Q_k$ and $R_k$ are symmetric real matrices. Therefore, in view of the Weyl inequalities (see equation (1.54) of \cite{taoTopicsRandomMatrix}), the eigenvalues of $Q_k+R_k$ can differ at most by constant times $k^2/n=o(1)$ from the eigenvalues of $Q_k$. We conclude the proof by noting that in view of Lemma \ref{lem:orthogonality} the eigenvalues of $Q_k$ are bounded from below by $c$ and from above by $C$ for $n$ large enough, hence the Gram matrix $Q_k+R_k$ also satisfies 
\[
  \frac{1}{2} c \II_k \leq \left( Q_k + R_k \right) \leq 2 C \II_k.
\]
This means that the rescaled basis $\sqrt{k}\phi$ satisfies the near orthogonality requirement, see Assumption \ref{def:DNNBasisorthonormal}. 
\end{proof}

\begin{lem}\label{lem:SplineAproxBounded}
  The DNN basis given in \eqref{eq:dnn:basis:Suzuki} satisfies that $\|\phi_{j} \|_{\infty}\leq 2.$ 
\end{lem}

\begin{proof}
  In view of Lemma \ref{lem:orthogonality} the cardinal  B-splines are bounded in supnorm by $1$. This implies our statement by~\eqref{eq:dnn:basis:Suzuki} and applying the triangle inequality.
\end{proof}

\section{Cardinal B-splines}\label{sec:B_splines}
One of the key steps in the proof of Lemma \ref{Lem:BasisProperties} is to use approximation of B-splines with deep neural networks, derived in \cite{suzukiAdaptivityDeepReLU2018}. In this chapter we collect properties of the cardinal B-splines used in our analysis. More specifically we show that they can be used to approximate functions in Besov spaces and we verify that they form a bounded, near orthogonal basis. 

We start by defining cardinal B-splines of order $q$ in $[0,1]$ and then extend the definition with tensors to the $d$-dimensional unit cube. Given $J+1$ knots $0=t_0<t_1<...<t_J=1$, the function $f: [0,1]\mapsto \mathbb{R}$ is a spline of order $q$ if its restriction to the interval $[t_{i},t_{i+1}]$, $i=0,...,J-1$ is a polynomial of degree at most $q-1$ and $f\in C^{q-2}[0,1]$ (provided that $q\geq 2$). For simplicity we will consider equidistant knots, i.e. $t_i=i/J$, $i\in\{0,...,J\}$, but our results can be extended to a more general knot structure as well. 

Splines form a linear space and a convenient basis for this space are given by B-splines $B_{1,q},...,B_{J,q}$. B-splines are defined recursively in the following way. First let us introduce additional knots at the boundary $t_{-q+1}=...=t_{-1}=t_{0}=0$ and $t_{J}=t_{J+1}=...=t_{J+q-1}$. Then we define the first order B-spline basis as $B_{j,1}(x)=1_{t_j\leq x<t_{j+1}}$, $j=0,...,J-1$. For higher order basis we use the recursive formula
$$B_{j,q}(x)=\frac{x-t_j}{t_{j+q-1}-t_j}B_{j,q-1}(x)+ \frac{t_{j+q}-x}{t_{j+q}-t_{j+1}}B_{j+1,q-1}(x),\quad j=-q+1,...,J-1.$$
From now on for simplicity we omit the order $q$ of the B-splines from the notation, writing $B_1,...,B_J$. We extend B-splines to dimension $d$ by tensorisation. For  $x \in [0,1]^d$ the $d$-dimensional cardinal B-splines are formed by taking the product of one dimensional B-splines, i.e. for  $j\in\{1,...,J\}^d$ and $x\in[0,1]^d$ we define 
\begin{align}
B_{j}(x) = \prod_{\ell = 1}^d B_{j_\ell}(x_\ell).\label{eq: B:splines}
\end{align}
\begin{remark}
We note that the $d$-dimensional index $j\in \{1,...,J\}^d$ can be replaced by a single index running from $1$ to $J^d =: k$. In this section for convenience we work with the multi-index formulation, but in the rest of the paper we consider the single index formulation.
\end{remark}

Next we list a few key properties of $d$-dimensional cardinal B-splines used in our proofs. In view of Chapter 12 of \cite{schumakerSplineFunctionsBasic2007} (see Definition 12.3 and Theorems 12.4-12.8) and Lemma  E.7 of \cite{ghosalFundamentalsNonparametricBayesian2017}, the cardinal B-splines have optimal approximation properties in the following sense.

\begin{lem}
  \label{thm:ApproximationPowerBySplines}
  Let $S$ be the space spanned by the cardinal B-splines of order $q\geq \beta$. 
  Then there exists a constant $C>0$ such that for all $f \in S^\beta_d(M)$ and all integers $\alpha \leq \beta$
  \[
    d(f, S) = \inf_{s \in S} \| f - s \|_2 \leq C k^{-\beta/d} \sum_{l = 1}^{d}  \left\| \frac{\partial^\alpha f}{\partial x_l^\alpha} \right\|_2,
  \]
with $k=J^d$. Moreover, if $\|f \|_\infty < F$, then one can pick $s$ such that $\| s \|_\infty< F$.
\end{lem}

Next we show that cardinal B splines are near orthogonal. The one dimensional case was considered in Lemma E.6 of \cite{ghosalFundamentalsNonparametricBayesian2017}. Here we extend these results to dimension $d$. Note that by tensorisation we will have $k = J^d$ spline basis functions.

\begin{lem}\label{lem:orthogonality}
  Let us denote by $B=(B_{j})_{j\in\{1,...,J\}^d}$ the collection of B-splines and by $\theta=(\theta_j)_{j\in\{1,...,J\}^d}$ the corresponding coefficients. Let $k= J^d$. Then there exists constant $c\in(0,1)$ such that
  \begin{align*}
    &c\| \theta \|_\infty \leq          \| \theta^T B \|_\infty     \leq \| \theta \|_\infty, \\
    &c\| \theta \|_2      \leq  \sqrt{k}\| \theta^T B \|_2          \leq \| \theta \|_2.
  \end{align*}
\end{lem}

\begin{proof}
  The bounds for the supremum norm follow from Lemma 2.2 of \cite{dejongeAdaptiveEstimationMultivariate2012}, hence it remained to deal with the bounds for the $\ell_2$-norm.

Let $I_{i}$, $i\in\{1,...,J\}^d$, denote the hypercube $\prod_{\ell = 1}^d [(i_\ell - 1)/J, i_\ell/J ]$ and $C_{i}$, $i\in\{1,...,J\}^d$ the collection of B-splines $B_j$, $j\in\{1,...,J\}^d$ which attain a nonzero value on the corresponding hypercube $I_{i}$. Then
  \[
    \| \theta^T B \|_2^2 = \int_{[0,1]^d} ( \theta^T B(x))^2 \dd x = \sum_{i\in\{1,...,J\}^d } \int_{I_{i}} \left(\sum_{j \in C_{i}} \theta_j B_j(x) \right)^2 \dd x.
  \]

Note that for a one-dimensional cardinal B-spline of degree $q$ we can distinguish $(2q-1)$ different cases, i.e. if $q\leq i_\ell\leq J-q$, the 1 dimensional splines are just translations of each other. Since the $d$-dimensional B-splines are defined as a tensor product of $d$ one-dimensional B-splines the number of distinct cases is $(2q-1)^d$. 

Define the translation map $T_{i}: \mathbb{R}^d\mapsto \mathbb{R}^d$, $i\in\{1,...,J\}^d$, to be the map given by $T_{i}(x) = (\frac{x_1 - i_1+ 1}{J}, \dots, \frac{x_d - i_d + 1}{J})$, then $\det T_{i} = J^{-d}$. This maps into the same space of polynomials regardless of $i$. This means
  \begin{align*}
    \int_{I_{i}} \left( \sum_{j \in C_{i}} \theta_j B_j(x) \right)^2 \dd x = J^{-d} \int_{[0,1]^d} \left( \sum_{j \in C_{i}} \theta_j B_j(T_{i}(x)) \right)^2 \dd x.
  \end{align*}

  We argue per case now. On each of these hypercubes $I_{i}$, $i\in\{1,...,J\}^d$, the splines are locally polynomials. 
  Then the inverse of $T_{i}$ defines a linear map between the polynomials spanned by the splines and the space of polynomials $P$ of order $q$. 
  The splines define $q^d$ basis functions on our cubes $I_i$. 
  Observe that each of the linear maps $T_{i}$ map the B-spline basis functions $B_j$, $j\in\{1,...,J\}^d$ to the same space of polynomials. 
  Note that by~\cite[Theorem 4.5]{schumakerSplineFunctionsBasic2007} and the rescaling property the $1$ dimensional splines restricted to the interval $[\frac{i}{j}, \frac{i+1}{J}]$ are linearly-independent. By tensorisation it follows that the $B$-spline basis restricted to the hypercube $I_i$ provides a linearly indepedent polynomial basis, hence $\sum_{j \in C_{i}}\theta_j^2$ defines a squared norm of the functions $x\mapsto\sum_{j \in C_{i}} \theta_j B_j(x)$, $x\in I_i$. Since in finite dimensional real vector spaces all norms are equivalent this results in
  \begin{align*}
    \sum_{j \in C_{i}} \theta_j^2 &\asymp \int_{[0,1]^d} \left( \sum_{j \in C_{i}} \theta_j B_j(T_{i}(x)) \right)^2 \dd x\\
&= J^{d} \int_{I_i} \left( \sum_{j \in C_{i}} \theta_j B_{j}(x)\right)^2 \dd x. 
  \end{align*}

In view of the argument  above, there are at most $(2q-1)^d$ different groups of hypercubes, hence the above result can be extended to the whole interval $[0,1]^d$ as well (by taking the worst case scenario constants in the above inequality out of the finitely many one), i.e. 
\begin{align*}
      \sum_{i\in\{1,....,J\}^d} \sum_{j \in C_{i}} \theta_j^2 &\asymp J^d \sum_{i\in\{1,....,J\}^d} \int_{I_{i}} \left(\sum_{j \in C_{i}} \theta_j B_j(x) \right)^2 \dd x\\
&=J^d\int_{[0,1]^d} \left(\sum_{j \in \{1,....,J\}^d} \theta_j B_j(x) \right)^2 \dd x.
 \end{align*}
 Since every $j$ on the left hand side occurs at most $(2q-1)^d$ many times in the sum, this leads us to
  \[
      \| \theta \|^2_2 \asymp J^d \| \theta^T B \|_2^2,
  \]
for some universal constants concluding the proof of our statement.
\end{proof}



\section{Concentration rates and uncertainty quantification of the posterior distribution} \label{sec:general:results}

In this section we provide posterior contraction rates and lower bounds for the radius of the credible balls under general conditions. These results are then applied for the Empirical Bayes Deep Neural Network method in Section \ref{sec:proofs:main}.

\subsection{Coverage theorem - general form}\label{sec:gen:them}

In this section we first provide a general theorem on the size of credible sets based on sieve type of priors. This result can be used beyond the nonparametric regression model and is basically the adaptation of Lemma 4 of \cite{rousseauAsymptoticFrequentistCoverage2020} to the non-adaptive setting with fixed sieve dimension $k$, not chosen by the empirical Bayes method as in \cite{rousseauAsymptoticFrequentistCoverage2020}. This theorem is of separate interest, as it can be used for instance for extending our results to other models, including nonparametric classification.

We start by introducing the framework under which our results hold. We consider a general statistical model, i.e. we assume that our data $\mathbb{D}_n$ is generated from a distribution $\PP_{f_0}$ indexed by an unknown functional parameter of interest $f_0$ belonging to some class of functions $\mathcal{F}$. Let us consider $k=k_n$ (not necessarily orthogonal) basis functions $\phi_1,...,\phi_{k}\in\mathcal{F}$ and use the notation $f_{\theta}(x)=\sum_{i=1}^k \theta_i \phi_i(x)$. Then we define the class $\Theta_k=\{\sum_{i=1}^k \theta_i \phi_i, \theta_i\in\mathbb{R}, i=1,...,k\}\subset\mathcal{F}$ and equivalently we also refer to the elements of this class using the coefficients $(\theta_1,....,\theta_k)$. We note that $f_0$ doesn't necessarily belong to the sub-class $\Theta_k$.  


Furthermore, let us consider a pseudometric $d_n: \mathcal{F}\times \mathcal{F}\mapsto\mathbb{R}$ and take $\theta^o=\arg\inf_{\theta\in \Theta_k} d_n(f_\theta,f_0)$, i.e. the projection of $f_0$ to the space $\Theta_k$ is $f_{\theta^o}$ with $\theta^o=(\theta^o_1,...,\theta^o_k)^T\in \Theta_k$ denoting the corresponding coefficient vector. Let us also consider a metric $d:\, \Theta_k\times \Theta_k\mapsto \mathbb{R}$ on the $k$-dimensional parameter space $\Theta_k$. Finally, we introduce the notation $B_k(\bar{\theta},\eps,d)=\{\theta\in \Theta_k:\, d(\theta,\bar\theta)\leq \eps\}$ for the $\eps$-radius $d$-ball in $\Theta_k$ centered at $\bar\theta\in\Theta_k$ and $B(\tilde{f},\eps,d_n)=\{f\in\mathcal{F}:\, d_n(\tilde{f},f)\leq \eps\}$ for the $\eps$-radius $d_n$-ball centered at $\tilde{f}\in\mathcal{F}$.


The next theorem provides lower bound for the radius $r_{n,\alpha}$ of the credible balls 
\[
  B(f_{\hat\theta},  \delta_\eps \eps_n,d_n )=\{f\in\mathcal{F}:\, d_n(f_{\theta},f_{\hat\theta})\leq r_{n,\alpha}\}
\]
 centered around an estimator $f_{\hat\theta}$. The radius $r_{n,\alpha}$ is defined as
\begin{align*}
\Pi_{k}\Big(\theta\in \Theta_k:\, f_\theta\in B(f_{\hat\theta},  r_{n,\alpha},d_n ) | \mathbb{D}_n \Big)=1-\alpha.
\end{align*}
Before stating the theorem we introduce some assumptions.

\begin{enumerate}[label=\textbf{A\arabic*}]
  \item \label{Ass:Centering_point} The centering point $f_{\hat\theta} \in \mathcal{F}$ satisfies that for all $\epsilon>0$ there exists $M_\epsilon>0$
  \begin{equation*}
    \sup_{f_0 \in  \mathcal{F}} \PP_{f_0} \left( d_n( f_{\theta^o} , f_{\hat\theta})\leq  M_\epsilon \sqrt{k/n}\right) \geq 1-\eps. 
  \end{equation*}
\item  \label{Ass:metrics} Assume that there exists $C_m>0$ such that for all $\theta,\theta'\in\Theta_k$
$$ C_m^{-1}d(\theta,\theta')\leq d_n(f_{\theta},f_{\theta'})\leq C_m d(\theta,\theta'). $$
 
  \item \label{Ass:UQ_assumptions} Assume that for all $M, \epsilon>0$ there exist constants $c_1, c_2, c_3, c_4, \delta_0, B_\epsilon > 0$ and $r \geq 2$ such that the following conditions hold
  \begin{enumerate}[label=\textbf{A\arabic{enumi}.\roman*}]
    \item \label{Ass:Metric_balls_contained_in_KL_balls}
    \begin{equation*}
      \begin{split}
        B_{k}(\theta^o,\sqrt{k/n},d ) \subset S_n(k, c_{1}, c_2,r),
      \end{split}
    \end{equation*}
    where $S_n(k,  c_{1}, c_2, r) =  $
    \[
      \hspace{-1cm}
      \Big\{
        \theta \in \Theta_{k}:\,
        \EE_{f_0} \log\frac{p_{\theta^o}}{p_{\theta}}  \leq c_{1} k,
        \EE_{f_0}\Big(\log\frac{p_{\theta^o}}{p_{\theta}}-\EE_{f_0}\log\frac{p_{\theta^o}}{p_{\theta}}\Big)^r\leq  c_2 k^{r/2}
       \Big\}.
    \]
  \item \label{Ass:Likelihood_ratio_control}Let $\bar{B}_k=\Theta_{k} \cap B( f_{\theta^o}, M_\epsilon \sqrt{k/n},d_n)$. Then for every $f_0\in  \mathcal{F}$
    \begin{equation*}
      \PP_{f_0}\Big(\sup_{f_\theta\in \bar{B}_k }\ell_n(\theta) - \ell_n(\theta^o)\leq  B_\eps k\Big)\geq 1-\eps,
    \end{equation*}
where $\ell_n(\theta)$ denotes the log-likelihood corresponding to the functional parameter $f_{\theta}\in\Theta_k$.
  \item \label{Ass:Small_ball_control}For every $\delta_0$ small enough
    \begin{equation*}
     \frac{   \sup_{\theta\in  \bar{B}_k } \Pi_{k}\big( B_k(\theta, \delta_{0}\sqrt{k/n},d)\big)}{\Pi_{k}\big( B_{k}(\theta^o, \sqrt{k/n},d )\big)} \leq c_{4} e^{ c_3 k \log (\delta_0) }.
    \end{equation*}

  \end{enumerate}
\end{enumerate}

\begin{thm}\label{thm:general:coverage}
Assume that conditions \ref{Ass:Centering_point}- \ref{Ass:UQ_assumptions} hold. Then for every $\eps>0$ there exists a small enough $\delta_\eps>0$ such that
 \begin{equation*}
 \sup_{f_0\in \mathcal{F}} \EE_{f_0} \Pi_{k}\left(\theta \in \Theta_{k}:\,d_n(f_\theta,f_{\hat\theta})\leq \delta_\eps  \sqrt{k/n}|\mathbb{D}_n \right) \leq \eps.
 \end{equation*}
\end{thm}

\begin{proof}
First note that since $\Pi_k$ is supported on $\Theta_k$,
  \begin{equation}\label{eq: UB:small:ball}
  \begin{split}
 \Pi_{k}\left(\theta:\, d_n(f_\theta,f_{\hat\theta})\leq \delta_\eps\sqrt{k/n}|\mathbb{D}_n  \right)&=    \frac{\int_{B(f_{\hat\theta},  \delta_\eps \sqrt{k/n},d_n)\cap \Theta_k}e^{\ell_n(\theta)-\ell_n(\theta^o)} d\Pi_{k}(\theta) }{\int_{\Theta_{k}}e^{\ell_n(\theta)-\ell_n(\theta^o)} d\Pi_{k}(\theta)}.
   \end{split}
  \end{equation}

Next let us introduce the notations
\begin{align}  
\Omega_{n}(C)=\left\{ e^{Ck}\frac{\int_{\Theta_{k}} e^{\ell_n(\theta) - \ell_n(\theta^o)}d\Pi_{k}(\theta)}{\Pi_{k}\big( B_{k}(\theta^o,\sqrt{k/n},d)\big)}\geq 1\right\},\\
\Gamma_n(B) = \sup_{B(f_{\hat\theta},  \delta_\eps \sqrt{k/n},d_n)\cap \Theta_k }\ell_n(\theta) - \ell_n(\theta^o )<Bk.
\end{align}
Note that in view of Assumptions \labelcref{Ass:Likelihood_ratio_control} and \ref{Ass:Centering_point} we have that $\inf_{f_0} \PP_{f_0}(\Gamma_n(B_{\eps}) ) \geq 1 - 2\epsilon$ for some large enough constant $B_\eps>0$ and in view of Assumption \labelcref{Ass:Metric_balls_contained_in_KL_balls} by using the standard technique for lower bound for the likelihood ratio (\cite[Lemma 8.37]{ghosalFundamentalsNonparametricBayesian2017}) we have  with $\PP_{f_0}$-probability bounded from below by $ 1-\eps$ that there exists $c_0>0$ such that
\begin{align*}
\int_{\Theta_k} e^{\ell_n(\theta) - \ell_n(\theta^o)}d\Pi_{k}(\theta) &\geq e^{-(c_0+1/\sqrt{\epsilon})k}\Pi_{k}\big(S_n(k,c_{1},c_2,r)\big)\nonumber\\
&\geq e^{-(c_0+1/\sqrt{\epsilon}) k_n }\Pi_{k}\big(B_{k}(\theta^o,\sqrt{k/n},d)\big),
  \end{align*}
hence $\PP_{f_0}( \Omega_n(c_0+1/\sqrt{\epsilon}) )\geq 1-\eps$.

Therefore, in view of assumption \ref{Ass:metrics}, the right hand side of \eqref{eq: UB:small:ball} is bounded from above on $A_n=\Omega_n(c_0+1/\sqrt{\epsilon})\cap \Gamma_n(B_{\eps})\cap \{d_n( f_{\theta^o} , f_{\hat\theta})\leq   M_\epsilon \sqrt{k/n}\} $ by
\begin{align*}
&e^{(B_{\eps}+c_{0}+1/\sqrt{\epsilon})k} \frac{ \Pi_{k}\left( \Theta_{k}\cap B(f_{\hat\theta},  \delta_\eps \sqrt{k/n},d_n )\right) }{ \Pi_{k}\left(B_{k}(\theta^o,   \sqrt{k/n},d) \right)}\\
&\qquad\qquad\leq e^{(c_{0}+B_{\eps}+1/\sqrt{\epsilon})k} \frac{ \Pi_{k}\big( B_k( \hat\theta,  C_m\delta_\eps \sqrt{k/n},d )\big) }{ \Pi_{k}\left(B_{k}(\theta^o,   \sqrt{k/n},d) \right)}\\
&\qquad\qquad\leq C  e^{\big(c_{0}+B_{\eps}+1/\sqrt{\epsilon}+c_{3}\log(C_m \delta_{\eps}) \big)  k}\leq   \eps,
\end{align*}
for small enough choice of $\delta_\eps>0$,  where the last line follows from assumption \labelcref{Ass:Small_ball_control} (with $\delta_{0}=C_m\delta_{\eps}$). Furthermore, note that
$$\PP_{f_0}(A_n^c)\leq \PP_{f_0}\big(\Omega_n(c_0+1/\sqrt{\epsilon})\big)
+\PP_{f_0}\big(\Gamma_n(B_{\eps}) \big)+ \PP_{f_0}\big( d_n( f_{\theta^o} , f_{\hat \theta})\geq M_\epsilon \sqrt{k/n} \big)\leq 4\eps,$$
where for the last term we used Assumption \ref{Ass:Centering_point}. Hence the $\EE_{f_0}$-expected value of the first term on the right hand side of \eqref{eq: UB:small:ball} is bounded from above by $5\eps$. 
\end{proof}

\subsection{Coverage in nonparametric regression}\label{sec:thm:regression}
We apply Theorem \ref{thm:general:coverage} in context of the uniform random design nonparametric regression model, i.e. we observe pairs of random variables $\mathbb{D}_n=\{(X_1,Y_1),..., (X_n,Y_n)\}$ satisfying that
\begin{align}
Y_i=f_0(X_i)+\eps_i,\qquad X_i\stackrel{iid}{\sim} \text{Unif}([0,1]^d),\quad \eps_i\stackrel{iid}{\sim}N(0,1),\quad i=1,...,n,\label{def:random:design:regression}
\end{align}
for some unknown functional parameter $f_0\in\mathcal{F}=L_2([0,1]^d,M)$. Let us denote by $X$ the collection of design points, i.e. $X=(X_1,...,X_n)$.

Let us consider $k=k_n$ (not necessarily orthogonal) basis functions $\phi_1,...,\phi_{k}\in\mathcal{F}$ and use the notation $f_{\theta}(x)=\sum_{i=1}^k \theta_i \phi_i(x)$. We denote by $\Phi_{n,k}$ the empirical basis matrix consisting the  basis functions $\phi_1,...,\phi_{k}$ evaluated at the design points $X_1,...,X_n$, i.e.
\begin{equation*}
\Phi_{n,k}= 
\begin{pmatrix}
\phi_1(X_1) & \phi_2(X_1) & \cdots & \phi_{k}(X_1) \\
\phi_1(X_2) & \phi_2(X_2) & \cdots & \phi_{k}(X_2) \\
\vdots  & \vdots  & \ddots & \vdots  \\
\phi_1(X_n) & \phi_2(X_n) & \cdots & \phi_{k}(X_n)
\end{pmatrix}.
\end{equation*}
Furthermore, let us denote the Gram matrix of the basis functions $\phi_1,...,\phi_k$ with respect to the $L_2$ inner product $\langle f,g \rangle=\int_{[0,1]^d} f(x) g(x) dx$ by
\begin{equation}\label{def:gram:mtx}
\Sigma_k=
\begin{pmatrix}
\langle \phi_1,\phi_1\rangle & \langle \phi_1,\phi_2\rangle & \cdots & \langle \phi_1,\phi_k\rangle \\
\langle \phi_2,\phi_1\rangle &\langle \phi_2,\phi_2\rangle  & \cdots & \langle \phi_2,\phi_k\rangle \\
\vdots  & \vdots  & \ddots & \vdots  \\
\langle \phi_k,\phi_1\rangle&\langle \phi_k,\phi_2\rangle & \cdots &\langle \phi_k,\phi_k\rangle
\end{pmatrix}.
\end{equation}
Finally, we need to impose the following near orthogonality assumption on the basis functions $\phi_1,...,\phi_k$.
\begin{enumerate}[label=\textbf{B\arabic*}]
  \item \label{Ass:Equivalence_of_metrics} Assume that there exists a constant $c_m \geq 1$ such that
$$c_m^{-1} I_k\leq  \Sigma_k\leq c_m I_k.$$
\end{enumerate}

In our analysis we consider a prior $\Pi_k$ supported on functions of the form $f_{\theta}=\sum_{j=1}^{k}\theta_j \phi_j$. We take priors of the product form, i.e.
\[
  d\Pi_k(\theta)=\prod_{j=1}^k g(\theta_j)d\theta,
\]
for a one dimensional density $g$, satisfying for every $M'> 0$ that there exists constants $\underline{c}, \overline{c} > 0$ such that
\begin{align}
   \underline{c} \leq g(x) \leq \overline{c},\qquad x\in[-M',M']\label{assump:prior}
\end{align}

\begin{lem}\label{thm:cov:reg}
  Consider the nonparametric regression model \eqref{def:random:design:regression} and a prior $\Pi_{k}$ satisfying assumption \eqref{assump:prior}. We assume that the Gram matrix $\Sigma_k$ given in \eqref{def:gram:mtx}, consisting bounded basis functions $\phi_j$, $j=1,...,k$, satisfies \ref{Ass:Equivalence_of_metrics} and that $\sum_{j=1}^k\phi_j(x)^2\leq Ck$, for all $x\in[0,1]^d$.  Furthermore, assume that the centering point of the credible set satisfies assumption \textbf{A1}. Then for every $\eps>0$ there exists a small enough $\delta_\eps>0$ such that
 \begin{equation*}
   \sup_{f_0\in L(([0,1]^d,M))} \EE_{f_0} \Pi_{k}\left(\theta\in\mathbb{R}^k:\,\|f_{\theta}-f_{\hat\theta}\|_2\leq \delta_\eps  \sqrt{k/n}|\mathbb{D}_n \right) \leq \eps.
 \end{equation*}
\end{lem}

\begin{proof}
  We show below that the conditions of Theorem \ref{thm:general:coverage} hold in this model for the conditional probability given the design points $\PP_{f_0}^{|X}(\cdot)=\PP_{f_0}(\cdot|X)$, on an event $A_n\subset \mathcal{X}^n$, where $\mathcal{X}=[0,1]^d$, satisfying $\PP_{f_0}(A_n^c)\leq \eps$, taking $d_n$ to be the empirical $L_2$ semi-metric, i.e. $d_n(f,g)^2=\|f-g\|_n^2=\sum_{i=1}^n \big(f(X_i)-g(X_i)\big)^2$ and $d$ the $\ell_2$-metric in $\Theta_k=\mathbb{R}^k$ i.e. $d(\theta,\theta')= \|\theta-\theta'\|_2$. Hence in view of Theorem \ref{thm:general:coverage}, on the event $A_n$ for every $\eps>0$ there exists a small enough $\delta_\eps>0$ such that
\begin{align*}
 \sup_{f_0\in \mathcal{F}} \EE_{f_0}^{|X} \Pi_{k}\left(\theta\in\mathbb{R}^k:\,\|f_\theta-f_{\hat\theta}\|_n\leq 2\delta_\eps  \sqrt{k/n}|\mathbb{D}_n \right) \leq \eps.
\end{align*}
Next note that in view of assertion \eqref{eq:reg:metrics}, see below, we get on an event $B_n$, with $\PP_{f_0}(B_n^c)\leq \eps$, that
\begin{align*}
\|f_\theta-f_{\hat\theta}\|_n/2\leq  \|f_{\theta}-f_{\hat\theta}\|_2,
\end{align*}
resulting in
\begin{align*}
 &\sup_{f_0\in \mathcal{F}} \EE_{f_0} \Pi_{k}\left(\theta:\,\|f_\theta-f_{\hat\theta}\|_2\leq \delta_\eps  \sqrt{k/n}|\mathbb{D}_n \right)\\
&\qquad\qquad\leq  \sup_{f_0\in \mathcal{F}} \EE_{f_0}\EE_{f_0}^{|X}\Big( \Pi_{k}\big(\theta:\,\|f_\theta-f_{\hat\theta}\|_n\leq 2\delta_\eps \sqrt{k/n}|\mathbb{D}_n\big) \Big)\\
& \qquad\qquad\leq \sup_{f_0\in \mathcal{F}} \PP_{f_0}(A_n^c)+ \PP_{f_0}(B_n^c)+\eps\leq 3\eps.
\end{align*}

It remained to prove that the conditions of  Theorem \ref{thm:general:coverage} hold.\\

\noindent\textbf{Condition \labelcref{Ass:Centering_point}.} Follows by the choice of the centering point.

\noindent\textbf{Condition \labelcref{Ass:metrics}.} 
First note that $\Sigma_{n,k}=n^{-1}\Phi_{n,k}^T\Phi_{n,k}$ has mean $\Sigma_{k}$. Then by the modified version of Rudelson's inequality \cite{rudelsonRandomVectorsIsotropic1999} we get that
\begin{align*}
\EE_{f_0}\| \Sigma_{n,k}-\Sigma_k \|_2\leq C \sqrt{\frac{\log k}{n}}\EE_{f_0}(\|\boldsymbol\phi (X_1) \|_2^{\log n})^{1/\log n},
\end{align*}
with $\boldsymbol\phi (X_1)=\big(\phi_1(X_1),...,\phi_k(X_1)\big)^T$. Note that by the boundedness assumption $\sum_{j=1}^k\phi_j(x)^2\leq Ck$, $x\in[0,1]^d$, the right hand side of the preceding display is bounded from above by constant times $\sqrt{k\log (k)/n}=o(1)$ on an event $B_n$ with $\PP_{f_0}(B_n)$ tending to one. Therefore,
\begin{align*}
 \Big|\|f_{\theta}\|_2^2-\|f_{\theta}\|_n^2\Big|&= \Big| \theta^{T} (\Sigma_k-\Sigma_{n,k})\theta\Big|
\leq \|\Sigma_k-\Sigma_{n,k}\|_2 \|\theta\|_2^2=o_{\PP_{f_0}}(\|\theta\|_2^2). 
\end{align*}
Furthermore, in view of Assumption \textbf{B1}
\[
  c_m^{-1} \|\theta\|_2^2 \leq  \|f_{\theta}\|_2^2= \theta^T\Sigma_k\theta\leq c_m \|\theta\|_2^2,\qquad\text{for all $\theta\in\mathbb{R}^k$},
\]
which in turn implies that on $B_n$
\begin{align}
(2c_m)^{-1} \|\theta\|_2^2\leq  \|f_{\theta}\|_2^2/2 \leq \|f_{\theta}\|_n^2\leq 2 \|f_{\theta}\|_2^2\leq 2c_m \|\theta\|_2^2,\label{eq:reg:metrics}
\end{align}
holds for all $\theta\in\mathbb{R}^k$.

\noindent\textbf{Condition \labelcref{Ass:Metric_balls_contained_in_KL_balls}.}
First note that for arbitrary $\theta\in\mathbb{R}^k$
\begin{align}
\ell_n(\theta)-\ell_n(\theta^o)&=
\frac{1}{2}\sum_{i=1}^n(Y_i-f_{\theta^o}(X_i))^2- \frac{1}{2}\sum_{i=1}^n(Y_i-f_{\theta}(X_i))^2\nonumber\\
 &=-\sum_{i=1}^n\Big((f_{\theta}(X_i)-f_{\theta^o}(X_i))^2/2- (Y_i-f_{\theta^o}(X_i))(f_{\theta}(X_i)-f_{\theta^o}(X_i))\Big)\nonumber\\
 &=-\sum_{i=1}^n\big(f_{\theta}(X_i)-f_{\theta^o}(X_i)\big)^2/2- \sum_{i=1}^n\eps_i(f_{\theta}(X_i)-f_{\theta^o}(X_i))\nonumber\\
&\qquad- \sum_{i=1}^n(f_0(X_i)-f_{\theta^o}(X_i))(f_{\theta}(X_i)-f_{\theta^o}(X_i))\nonumber\\
&=-\sum_{i=1}^n\big(f_{\theta}(X_i)-f_{\theta^o}(X_i)\big)^2/2- \sum_{i=1}^n\eps_i\big(f_{\theta}(X_i)-f_{\theta^o}(X_i)\big), \label{eq: help:exp:term}
\end{align}
where in the last line we used that $f_{\theta^o}$ is the orthogonal projection of $f_0$ to $\Theta_k=\{\sum_{i=1}^k \theta_i \phi_i:\, \theta\in\mathbb{R}^k\}$ with respect to the empirical Euclidean norm $d_{n}$. Then by taking $\EE_{f_0}^{|X}$-expectation on both sides of \eqref{eq: help:exp:term} we get that

\begin{align*}
\EE_{f_0}^{|X}\big(\ell_n(\theta)-\ell_n(\theta^o)\big)=n\|f_{\theta}- f_{\theta^o}\|_n^2/2\leq c_m n\|\theta-\theta^o\|_2^2.
\end{align*}
Similarly $\EE_{f_0}^{|X} \big[ \ell_n(\theta)-\ell_n(\theta^o)- \EE_{f_0}^{|X}(\ell_n(\theta)-\ell_n(\theta^o)) \big]^2\leq 2c_m n\|f_{\theta}- f_{\theta^o}\|_2^2$,
hence $B_k(\theta^o,\sqrt{k/n},\|\cdot\|_2)\subset \mathcal{S}_n(k,c_m,2c_m,2)$.\\

\noindent\textbf{Condition \labelcref{Ass:Likelihood_ratio_control}.}
In view of assertion \eqref{eq: help:exp:term} and using Cauchy-Schwarz inequality (as in  inequality (A.3) of the supplementary material of \cite{rousseauAsymptoticFrequentistCoverage2020} we arrive at
\begin{align*}
\ell_n(\theta)-\ell_n(\theta^o)&=
- n\|f_{\theta}-f_{\theta^o}\|_n^2/2- \eps^T \Phi_{n,k}(\theta-\theta^o)  \\
&\leq- n\|f_{\theta}-f_{\theta^o}\|_n^2/2+\|\eps^T \Phi_{n,k}\|_2\|\theta^o-\theta\|_2.
\end{align*}
We show below that with $\PP_{f_0}^{|X}$-probability tending to one
\begin{align}
\|\eps^T\Phi_{n,k}\|_2^2\leq C kn. \label{eq: UB:cross:term}
\end{align}
Hence on the same event we get that
\begin{align*}
\ell_n(\theta)-\ell_n(\theta^o)&\leq \sqrt{n} \|f_{\theta}-f_{\theta^o}\|_n\Big(C\sqrt{c_m} M_\eps \sqrt{k}-\sqrt{n}\|f_{\theta}-f_{\theta^o}\|_n/2  \Big)\\
&\leq 2c_mC^2M_{\eps}^2k.
\end{align*}

It remained to prove that \eqref{eq: UB:cross:term} holds with probability tending to one. Note that in view of assertion \eqref{eq:reg:metrics} on an event $B_n$, with  $\PP_{f_0}(B_n)\rightarrow 1$ we get for $\eps\sim N_k(0,I_k)$
\begin{align*}
\|\eps^T \Phi_{n,k}\|_2^2=n \eps^T\Sigma_{n,k} \eps\leq 2c_m n\|\eps\|_2^2.
\end{align*}
Then by the properties of the $\chi^2_k$ distribution the right hand side of the preceding display is bounded from above by $4 c_m^2 nk$ with probability tending to one as $k$ tends to infinity.\\

\noindent\textbf{Condition \labelcref{Ass:Small_ball_control}}
In view of assertion \labelcref{eq:reg:metrics} on the event $B_n$ we get that
\begin{align*}
 B(f_\theta,\delta_0\sqrt{k/n},\|\cdot\|_n)\subset B_k(\theta,\sqrt{2c_m}\delta_0\sqrt{k/n},\|\cdot\|_2),\\  B(f_{\theta^o},\sqrt{k/n},\|\cdot\|_n)\supset B_k(\theta^o,\sqrt{k/n}/\sqrt{2c_m},\|\cdot\|_2).
\end{align*}
Next note that in view of \eqref{eq:reg:metrics} on an event $B_n$ with $\PP_{f_0}(B_n)\rightarrow 1$ we have
\begin{align}
(2c_m)^{-1}\|\theta^o\|_2^2\leq \|f_{\theta^o}\|_n^2\leq \|f_0\|^2_n\leq M,\label{eq:help:prior:mass:metric}
\end{align}
resulting in $|\theta^o_j|\leq \sqrt{2c_m M}< M'$. Therefore the density of the prior is bounded away from zero and infinity on a neighbourhood of $\theta^o$ which in turn implies that
\begin{align*}
\sup_{\theta\in B(f_{\theta^o},M,\|\cdot\|_n)}\frac{\Pi_k \big(B(f_\theta,\delta_0\sqrt{k/n},\|\cdot\|_n) \big)}{\Pi_k \big(B(f_{\theta^o},\sqrt{k/n},\|\cdot\|_n) \big)}
&\lesssim \frac{ \text{Vol}\big(B_k(\theta,\sqrt{2c_m}\delta_0\sqrt{k/n},\|\cdot\|_2)\big)}{\text{Vol}\big(B_k(\theta^o,\sqrt{k/n}/\sqrt{2c_m},\|\cdot\|_2)\big)}\\
&\lesssim e^{c k \log (2c_m\delta_0) }.
\end{align*}

\end{proof}

\subsection{Misspecified contraction rates}

Finally, we derive a contraction rate result for the posterior in our misspecified setting. We assume that our true model parameter is $f_0$, which however, does not necessarily belong to our model $\Theta_k=\{f_{\theta}=\sum_{i=1}^k \theta_i\phi_i:\, \theta\in\mathbb{R}^k\}$. Let us denote by $f^*=f_{\theta^*}$ the $L_2$-projection of $f_0$ into the subspace $\Theta_k$. We show below that the posterior contracts with the $L_2$-rate $\sqrt{k/n}$ around $f^*$ in the regression model.

\begin{thm}\label{thm:contraction}
Consider the random design nonparametric regression model with observations $\mathbb{D}_n=\big((X_1,Y_1),...,(X_n,Y_n)\big)$ and assume that the Gram matrix
$\Sigma_{k}$ given in \eqref{def:gram:mtx} satisfies Assumption \ref{Ass:Equivalence_of_metrics}
and that the prior $\Pi_{k}$ satisfies \eqref{assump:prior}. Then
   \[
    \limsup_{n \rightarrow \infty}
    \sup_{f_0 \in S^{\beta}_d(M)\cap L_\infty(M)}
    \EE_{f_0}
    \Pi_k\big(\theta\in\mathbb{R}^k :\,
      \| f_{\theta} - f^* \|_2 \geq M_n \sqrt{k/n} | \mathbb{D}_n
    \big)
    = 0
  \]
  for all $M_n \rightarrow \infty$.
\end{thm}

\begin{proof}

For ease of notation we set $\eps_n=\sqrt{k/n}$. Then we show below that the following two inequalities hold for some constant $J>0$,
\begin{align}
\frac{\Pi_k\left(\theta:\, j \eps_n\leq \| f_{\theta} - f^* \|_2 \leq 2 j \eps_n \right)}
{\Pi_k\left( B(f^*, \eps_n,\|\cdot\|_2) \right)}
  \leq e^{ j^2k/8},\qquad \text{for all $j\geq J$},\label{eq: cond:1:contr}\\
 \log N(\epsilon,     \{f_\theta : \epsilon     < \|f^* - f_{\theta}\|_2 \leq 2 \epsilon \},\|\cdot\|_2) \leq k, \quad\text{for  all $\epsilon > 0$}.\label{eq: cond:2:contr}
\end{align}
The function class $\mathcal{F}=S^{\beta}_d(M)\cap L_{\infty}(M)$ is closed, convex, and uniformly bounded. Furthermore, since the Gaussian noise satisfies $\EE_{f_0}e^{M|\eps_i|}<\infty$ for all $M>0$, in view of Lemma 8.41 of \cite{ghosalFundamentalsNonparametricBayesian2017} condition (8.52) of \cite{ghosalFundamentalsNonparametricBayesian2017} holds.  Therefore, in view of Lemma 8.38 of \cite{ghosalFundamentalsNonparametricBayesian2017} the logarithm of the covering number for testing under misspecification is bounded from above by $\log N(\epsilon,     \{f_\theta : \epsilon     < \|f^* - f_{\theta}\|_2 \leq 2 \epsilon \},\|\cdot\|_2)$, which in turn is bounded by $k$ following from \eqref{eq: cond:2:contr}.
Then our statement follows by applying Theorem 8.36 of \cite{ghosalFundamentalsNonparametricBayesian2017} (with $\eps_n=\bar\eps_n=k/n$, $d=\|\cdot\|_2$, $\mathcal{P}_{n,1}=\mathcal{F}$, $\mathcal{P}_{n,2}=\emptyset$).

It remained to verify conditions \eqref{eq: cond:1:contr} and \eqref{eq: cond:2:contr}.\\

\noindent\textbf{Proof of \eqref{eq: cond:1:contr}}. First note that in view of condition \ref{Ass:Equivalence_of_metrics}
\begin{align*}
\frac{\Pi_k\big(\theta:\, j \eps_n\leq \| f_{\theta} - f^* \|_2 \leq 2 j \eps_n \big)}
{\Pi_k\big( B(f^*, \eps_n,\|\cdot\|_2) \big)}
\leq \frac{\Pi_k\big(B_k(\theta^*, 2j\sqrt{c_m}\eps_n,\|\cdot\|_2) \big)}
{\Pi_k\left( B_k(\theta^*, \eps_n/\sqrt{c_m},\|\cdot\|_2) \right)}.
\end{align*}
Furthermore, note that the prior density is bounded from above and below by  $\overline{c}^{k}$ and $\underline{c}^{k}$, respectively, in a neigbourhood of $f_{\theta^*}$ following from assumption \eqref{assump:prior} and by similar argument as in \eqref{eq:help:prior:mass:metric}. Therefore the prior probability of a given set $A$ can be upper and lower bounded by the Euclidean volume of $A$ times $\overline{c}^{k}$ and $\underline{c}^{k}$, respectively. This implies that the preceding display can be further bounded from above by
\begin{align*}
\Big( \frac{\overline{c}}{\underline{c}} \Big)^{k}\frac{\text{Vol}\Big(B_k(\theta^*, 2j\sqrt{c_m}\eps_n,\|\cdot\|_2) \Big)}{\text{Vol}\Big(B_k(\theta^*,\eps_n/\sqrt{c_m} ,\|\cdot\|_2) \Big)}\leq \Big( \frac{2c_mj\overline{c}}{\underline{c}} \Big)^{k} \leq e^{ j^2k/8},
\end{align*}
for all $j\geq J$, for $J$ large enough.\\

\noindent\textbf{Proof of \eqref{eq: cond:2:contr}}.
First note that by Assumption \ref{Ass:Equivalence_of_metrics} 
\begin{align*}
  \log
 & N(\epsilon,     \{f_\theta : \epsilon     < \| f_\theta - f^* \|_2 \leq 2 \epsilon \},\|\cdot \|_2)\\
&\qquad\qquad \leq
 \log   N(\epsilon  c_m^{-1/2}, \{  \theta : \epsilon c_m^{-1/2} < \| \theta - \theta^* \|_2 \leq 2 \epsilon c_m^{1/2} \}, \| \cdot \|_2 ).
\end{align*}
Then  in view of the standard bound on the local entropies of $k$-dimensional Euclidean balls the right hand side is further bounded by constant times $k$, finishing the proof of our statement.\\

\end{proof}

\end{document}